\documentclass[a4paper, english]{article}

\usepackage{colortbl}
\usepackage{color}
\usepackage[monochrome]{xcolor}

\usepackage{etex}
\usepackage{graphics}
\usepackage{cancel}
\usepackage{stmaryrd}
\usepackage{amsmath}
\usepackage{amsthm}
\usepackage{amssymb}
\usepackage{proof}
\usepackage{verbatim,cmll}
\usepackage{setspace}
\usepackage{lscape}
\usepackage{latexsym,relsize}
\usepackage{amscd}
\usepackage{bussproofs}
\usepackage[position=top]{subfig}
\usepackage{leqno}
\usepackage{tikz}
\usepackage{authblk}
\usepackage{multicol}
\usetikzlibrary{arrows}
\usetikzlibrary{matrix}
\usetikzlibrary{patterns}
\usetikzlibrary{shapes}
\usetikzlibrary{automata}

\usepackage[geometry]{ifsym}

\usepackage{pgf}
\usepackage{amsmath,amssymb}
\usepackage{mathrsfs}
\usepackage{bm}
\usepackage{bbm}
\usepackage{verbatim}
\usepackage{epsfig}
\usepackage{graphicx}
\usepackage{import}
\usepackage{bussproofs}
\usepackage{qtree}
\usepackage{tree-dvips}
\usepackage{lingmacros}
\usepackage{amsfonts}
\usepackage{mathdots}
\usepackage{amssymb}
\usepackage{pifont}
\usepackage[safe]{tipa}
\usepackage{newlfont}
\usepackage{latexsym}
\usepackage{multirow}
\usepackage{array,cmll}
\usepackage{textcomp}
\usepackage{accents}
\usepackage{rotating}
\usepackage{logreq}

\usepackage{hyperref}

\usepackage{amsmath}
\usepackage{amsthm}
\usepackage{latexsym}
\usepackage{amscd}
\usepackage{amssymb}
\usepackage{bussproofs}
\usepackage[position=top]{subfig}

\usepackage{leqno}
\usepackage{url}	
\usepackage{lipsum}
\usepackage{cleveref}
\usepackage{enumerate}
\usepackage{bpextra} 
\usepackage{centernot} 
\usepackage{xypic}
\usepackage{xspace}

\newcommand{\commment}[1]{}

\def\aol{\rule[0.5865ex]{1.38ex}{0.1ex}}

\def\pdra{\mbox{$\,>\mkern-8mu\raisebox{-0.065ex}{\aol}\,$}}
\def\pdla{\mbox{\rotatebox[origin=c]{180}{$\,>\mkern-8mu\raisebox{-0.065ex}{\aol}\,$}}}

\def\mANDORatom#1{\hbox{\hbox to 0pt{$#1\TriangleUp$\hss}$#1\TriangleDown$}}

\newcommand{\mcAND}{%
\mathrel{\ooalign{\raisebox{-0.39ex}{$\mbox{\TriangleUp}$}\cr\kern4.2pt{\raisebox{-0.13ex}{$\cdot$}}}}}
\newcommand{\mcand}{%
\mathrel{\ooalign{$\vartriangle$\cr\kern1.99pt{\raisebox{-0.17ex}{$\cdot$}}}}}

\newcommand{\nAND}{%
\mathrel{\ooalign{$\mbox{\TriangleUp}$\cr\kern0pt$\mbox{\rotatebox[origin=c]{180}{\TriangleUp}}$}}}

\newcommand{\nand}{%
\mathrel{\ooalign{$\vartriangle$\cr\kern0pt$\triangledown$}}}

\newcommand{\mcBAND}{%
\mathrel{\ooalign{\raisebox{-0.39ex}{$\mbox{\FilledTriangleUp}$}\cr\kern4.2pt{\raisebox{-0.13ex}{${\color{white}\cdot}$}}}}}

\newcommand{\mcband}{%
\mathrel{\ooalign{$\blacktriangle$\cr\kern1.99pt{\raisebox{-0.17ex}{${\color{white}\cdot}$}}}}}

\newcommand{\mcRA}{%
\mathrel{\ooalign{
                  \raisebox{-0.3ex}{$\rotatebox[origin=c]{-90}{$\mbox{{\TriangleUp}}$}$}
                                                                            \cr\kern2.7pt{\raisebox{0.2ex}{$\cdot\mkern1.3mu$}}}}}

\newcommand{\mcra}{%
\mathrel{\ooalign{$\,{\vartriangleright\,}$\cr\kern3pt{\raisebox{0ex}{$\cdot$}}}}}

\newcommand{\mcraline}{%
-{\mkern-6mu{\mathrel{\ooalign{$\,{\vartriangleright\,}$\cr\kern3pt{\raisebox{0ex}{$\cdot$}}}}}}}

\newcommand{\mdraline}{%
{\mathrel{\ooalign{$\,{\vartriangleright\,}$\cr\kern3pt{\raisebox{0ex}{$\cdot$}}}}}{\mkern-6mu}-}

\newcommand{\cra}{%
\mathrel{\ooalign{$\,-{\mkern-3mu\vartriangleright\,}$\cr\kern8pt{\raisebox{0ex}{$\cdot$}}}}}

\newcommand{\mcBRA}{%
\mathrel{\ooalign{
                  \raisebox{-0.3ex}{$\rotatebox[origin=c]{-90}{$\mbox{\FilledTriangleUp}$}$}
                                                                            \cr\kern2.7pt{\raisebox{0.2ex}{${\color{white}\cdot}$}}}}}

\newcommand{\mcbra}{%
\mathrel{\ooalign{$\,-{\mkern-3mu\blacktriangleright\,}$\cr\kern8pt{\raisebox{0ex}{$\cdot$}}}}}

\newcommand{\mcLA}{%
\mathrel{\ooalign{
                  \raisebox{-0.3ex}{$\rotatebox[origin=c]{90}{$\mbox{\TriangleUp}$}$}
                                                                                     \cr\kern5.5pt{\raisebox{0.2ex}{$\cdot$}}
                                                                                                                              }}}

\newcommand{\mcla}{%
\mathrel{\ooalign{$\,{\vartriangleleft\,}$\cr\kern5pt{\raisebox{0ex}{$\cdot$}}}}}

\newcommand{\mclaline}{%
-{\mkern-6mu{\mathrel{\ooalign{$\,{\vartriangleleft\,}$\cr\kern5pt{\raisebox{0ex}{$\cdot$}}}}}}}

\newcommand{\mcBLA}{%
\mathrel{\ooalign{
                  \raisebox{-0.3ex}{$\rotatebox[origin=c]{90}{$\mbox{\FilledTriangleUp}$}$}
                                                                                     \cr\kern5.5pt{\raisebox{0.2ex}{${\color{white}\cdot}$}}
                                                                                                                                            }}}

  \numberwithin{equation}{section}

\newcommand{\pla}{\leftarrow}
\newcommand{\aatop}{\ensuremath{\top}\xspace}
\newcommand{\abot}{\ensuremath{\bot}\xspace}
\newcommand{\aand}{\ensuremath{\wedge}\xspace}
\newcommand{\aor}{\ensuremath{\vee}\xspace}
\newcommand{\ararr}{\ensuremath{\rightarrow}\xspace}

\newcommand{\adrarr}{\ensuremath{\,{>\mkern-7mu\raisebox{-0.065ex}{\rule[0.5865ex]{1.38ex}{0.1ex}}}\,}\xspace}
\newcommand{\adlarr}{\ensuremath{\,{\raisebox{-0.065ex}{\rule[0.5865ex]{1.38ex}{0.1ex}} \mkern-7mu<}\,}\xspace}

\newcommand{\AATOP}{\hat{\top}}
\newcommand{\ABOT}{\ensuremath{\check{\bot}}\xspace}
\newcommand{\AAND}{\ensuremath{\:\hat{\wedge}\:}\xspace}
\newcommand{\AOR}{\ensuremath{\:\check{\vee}\:}\xspace}
\newcommand{\ARARR}{\ensuremath{\:\check{\rightarrow}\:}\xspace}
\newcommand{\ALARR}{\ensuremath\:\check{\leftarrow}\:\xspace}
\newcommand{\ADRARR}{\ensuremath{\hat{{\:{>\mkern-7mu\raisebox{-0.065ex}{\rule[0.5865ex]{1.38ex}{0.1ex}}}\:}}}\xspace}
\newcommand{\ADLARR}{\ensuremath{\,\hat{{\:\raisebox{-0.065ex}{\rule[0.5865ex]{1.38ex}{0.1ex}}\mkern-7mu< }\:}}\xspace}

\newcommand{\gI}{%
\mathrel{\ooalign{$\mbox{T}$\cr\kern0pt$\mbox{\rotatebox[origin=c]{180}{T}}$}}}

\newcommand{\gbot}{\rotatebox[origin=c]{180}{$\tau$}}

\def\aol{\rule[0.5865ex]{1.38ex}{0.1ex}}

\makeatletter
\newcommand{\WKnowProxy}[2]{%
  {\mathbin{\ooalign{$#1\circ#2 $\cr\hidewidth
   \raise.155ex\hbox{$#1{\scriptstyle{\ast}}#2$}\hidewidth\cr  }}}}
\makeatother

\makeatletter
\newcommand{\BKnowProxy}[2]{%
  {\mathbin{\ooalign{$#1\bullet#2 $\cr\hidewidth
   \raise.155ex\hbox{$#1{\scriptstyle{\color{white}{\ast}}}#2$}\hidewidth\cr  }}}}
\makeatother









\newcommand{\fns}{\footnotesize}
\newcommand{\mc}{\multicolumn}

\def\fCenter{{\mbox{$\ \vdash\ $}}}

\EnableBpAbbreviations

\renewcommand{\epsilon}{\varepsilon}

\newcommand{\nomj}{\mathbf{j}}

\tikzset{
	treenode/.style = {align=center, inner sep=0pt, text centered},
	Ske/.style = {treenode, ellipse, double, draw=black,
		minimum width=6pt, thick},
	PIA/.style = {treenode, ellipse, black, draw=black,
		minimum width=6pt},
	Crit/.style = {treenode, rectangle, draw=black,
		minimum width=0.5em, minimum height=0.5em}
}

\newcommand{\I}{\textrm{I}}

\newcommand{\gbulf}{\raisebox{0.06em}{$\textcolor{Black}{\scriptstyle(f)}$}}

\newcommand{\gdiaf}{\raisebox{0.06em}{$\textcolor{Black}{\scriptstyle\langle\,{f}\,\rangle}$}}

\newcommand{\gboxf}{\raisebox{0.06em}{$\textcolor{Black}{\scriptstyle\lbrack\,{f}\,\rbrack}$}}

\newcommand{\gbulpx}{\raisebox{0.06em}{$\textcolor{Black}{\scriptstyle(\pi_x)}$}}

\newcommand{\gdiapx}{\raisebox{0.06em}{$\textcolor{Black}{\scriptstyle\langle\,{\pi_x}\,\rangle}$}}

\newcommand{\gboxpx}{\raisebox{0.06em}{$\textcolor{Black}{\scriptstyle\lbrack\,{\pi_x}\,\rbrack}$}}

\newcommand{\gbulpy}{\raisebox{0.06em}{$\textcolor{Black}{\scriptstyle(\pi_y)}$}}

\newcommand{\gdiapy}{\raisebox{0.06em}{$\textcolor{Black}{\scriptstyle\langle\,{\pi_y}\,\rangle}$}}

\newcommand{\gboxpy}{\raisebox{0.06em}{$\textcolor{Black}{\scriptstyle\lbrack\,{\pi_y}\,\rbrack}$}}

\newcommand{\gbulpz}{\raisebox{0.06em}{$\textcolor{Black}{\scriptstyle(\pi_z)}$}}

\newcommand{\gdiapz}{\raisebox{0.06em}{$\textcolor{Black}{\scriptstyle\langle\,{\pi_z}\,\rangle}$}}

\newcommand{\gdiapzu}{\raisebox{0.06em}{$\textcolor{Black}{\scriptstyle\langle\,{\pi_{z_1}}\,\rangle}$}}

\newcommand{\gdiapzn}{\raisebox{0.06em}{$\textcolor{Black}{\scriptstyle\langle\,{\pi_{z_n}}\,\rangle}$}}

\newcommand{\gboxpz}{\raisebox{0.06em}{$\textcolor{Black}{\scriptstyle\lbrack\,{\pi_z}\,\rbrack}$}}

\newcommand{\gbultx}{\raisebox{0.06em}{$\textcolor{Black}{\scriptstyle(t_{\overline{x}})}$}}

\newcommand{\gbultf}{\raisebox{0.06em}{$\textcolor{Black}{\scriptstyle(t_F)}$}}

\newcommand{\gbulst}{\raisebox{0.06em}{$\textcolor{Black}{\scriptstyle(s_T)}$}}

\newcommand{\gdiatf}{\raisebox{0.06em}{$\textcolor{Black}{\scriptstyle\langle\,{t_F}\,\rangle}$}}

\newcommand{\gbulsytf}{\raisebox{0.06em}{$\textcolor{Black}{\scriptstyle(s_y, t_F)}$}}

\newcommand{\gdiasytf}{\raisebox{0.06em}{$\textcolor{Black}{\scriptstyle\langle\,s_y, t_F\,\rangle}$}}

\newcommand{\gdiast}{\raisebox{0.06em}{$\textcolor{Black}{\scriptstyle\langle\,s_T\,\rangle}$}}

\newcommand{\gbulzytf}{\raisebox{0.06em}{$\textcolor{Black}{\scriptstyle(z_y, t_F)}$}}

\newcommand{\gdiazytf}{\raisebox{0.06em}{$\textcolor{Black}{\scriptstyle\langle\,z_y, t_F\,\rangle}$}}

\newcommand{\gbulstxt}{\raisebox{0.06em}{$\textcolor{Black}{\scriptstyle(s(t_{\overline{x}}/\overline{x})_T)}$}}

\newcommand{\gdiastxt}{\raisebox{0.06em}{$\textcolor{Black}{\scriptstyle\langle\,s(t_{\overline{x}}/\overline{x})_T\,\rangle}$}}

\newcommand{\DIAX}{\raisebox{0.06em}{$\textcolor{Black}{\scriptstyle\langle\,\hat{x}\,\rangle}$}}
\newcommand{\diax}{\raisebox{0.06em}{$\textcolor{Black}{\scriptstyle\langle\,{x}\,\rangle}$}}

\newcommand{\diaxf}{\raisebox{0.06em}{$\textcolor{Black}{\scriptstyle\langle\,{x}\,\rangle}_{F}$}}
\newcommand{\DIAY}{\raisebox{0.06em}{$\textcolor{Black}{\scriptstyle\langle\,\hat{y}\,\rangle}$}}
\newcommand{\diay}{\raisebox{0.06em}{$\textcolor{Black}{\scriptstyle\langle\,{y}\,\rangle}$}}
\newcommand{\DIAZ}{\raisebox{0.06em}{$\textcolor{Black}{\scriptstyle\langle\,\hat{z}\,\rangle}$}}
\newcommand{\diaz}{\raisebox{0.06em}{$\textcolor{Black}{\scriptstyle\langle\,{z}\,\rangle}$}}

\newcommand{\DIAXU}{\raisebox{0.06em}{$\textcolor{Black}{\scriptstyle\langle\,\hat{x}_1\,\rangle}$}}

\newcommand{\BOXX}{\raisebox{0.06em}{$\textcolor{Black}{\scriptstyle\lbrack\,\check{x}\,\rbrack}$}}
\newcommand{\boxx}{\raisebox{0.06em}{$\textcolor{Black}{\scriptstyle\lbrack\,{x}\,\rbrack}$}}

\newcommand{\boxxf}{\raisebox{0.06em}{$\textcolor{Black}{\scriptstyle\lbrack\,{x\,\rbrack}}_{F}$}}
\newcommand{\BOXY}{\raisebox{0.06em}{$\textcolor{Black}{\scriptstyle\lbrack\,\check{y}\,\rbrack}$}}
\newcommand{\boxy}{\raisebox{0.06em}{$\textcolor{Black}{\scriptstyle\lbrack\,{y}\,\rbrack}$}}
\newcommand{\BOXZ}{\raisebox{0.06em}{$\textcolor{Black}{\scriptstyle\lbrack\,\check{z}\,\rbrack}$}}
\newcommand{\boxz}{\raisebox{0.06em}{$\textcolor{Black}{\scriptstyle\lbrack\,{z}\,\rbrack}$}}

\newcommand{\BOXXU}{\raisebox{0.06em}{$\textcolor{Black}{\scriptstyle\lbrack\,\check{x_1}\,\rbrack}$}}

\newcommand{\BOXXD}{\raisebox{0.06em}{$\textcolor{Black}{\scriptstyle\lbrack\,\check{x_2}\,\rbrack}$}}

\newcommand{\BOXXK}{\raisebox{0.06em}{$\textcolor{Black}{\scriptstyle\lbrack\,\check{x_k}\,\rbrack}$}}

\newcommand{\DIAXK}{\raisebox{0.06em}{$\textcolor{Black}{\scriptstyle\langle\,\hat{x_k}\,\rangle}$}}

\newcommand{\DIAXD}{\raisebox{0.06em}{$\textcolor{Black}{\scriptstyle\langle\,\hat{x_2}\,\rangle}$}}

\newcommand{\bulstar}{\raisebox{0.06em}{$\scriptstyle(\star)$}}
\newcommand{\DIASTAR}{\raisebox{0.06em}{$\textcolor{Black}{\scriptstyle\langle\,\hat{\star}\,\rangle}$}}

\newcommand{\BOXSTAR}{\raisebox{0.06em}{$\textcolor{Black}{\scriptstyle\lbrack\,\check{\star}\,\rbrack}$}}

\newcommand{\BULX}{\raisebox{0.06em}{$\textcolor{Black}{\scriptstyle(\tilde{x})}$}}
\newcommand{\bulx}{\raisebox{0.06em}{$\textcolor{Black}{\scriptstyle(x)}$}}

\newcommand{\BULV}{\raisebox{0.06em}{$\textcolor{Black}{\scriptstyle(\tilde{v})}$}}

\newcommand{\bulxt}{\raisebox{0.06em}{$\textcolor{Black}{\scriptstyle(x)}_{\scriptstyle T}$}}

\newcommand{\BULZD}{\raisebox{0.06em}{$\textcolor{Black}{\scriptstyle(\tilde{z}_2)}$}}

\newcommand{\BULY}{\raisebox{0.06em}{$\textcolor{Black}{\scriptstyle(\tilde{y})}$}}
\newcommand{\buly}{\raisebox{0.06em}{$\textcolor{Black}{\scriptstyle(y)}$}}
\newcommand{\BULZ}{\raisebox{0.06em}{$\textcolor{Black}{\scriptstyle(\tilde{z})}$}}
\newcommand{\bulz}{\raisebox{0.06em}{$\textcolor{Black}{\scriptstyle(z)}$}}

\newcommand{\BULXU}{\raisebox{0.06em}{$\textcolor{Black}{\scriptstyle(\tilde{x}_1)}$}}
\newcommand{\bulxu}{\raisebox{0.06em}{$\textcolor{Black}{\scriptstyle(x_1)}$}}
\newcommand{\BULXD}{\raisebox{0.06em}{$\textcolor{Black}{\scriptstyle(\tilde{x}_2)}$}}

\newcommand{\BULYU}{\raisebox{0.06em}{$\textcolor{Black}{\scriptstyle(\tilde{y}_1)}$}}
\newcommand{\bulyu}{\raisebox{0.06em}{$\textcolor{Black}{\scriptstyle(y_1)}$}}
\newcommand{\BULZU}{\raisebox{0.06em}{$\textcolor{Black}{\scriptstyle(\tilde{z}_1)}$}}
\newcommand{\bulzu}{\raisebox{0.06em}{$\textcolor{Black}{\scriptstyle(z_1)}$}}

\newcommand{\BULYM}{\raisebox{0.06em}{$\textcolor{Black}{\scriptstyle(\tilde{y}_m)}$}}
\newcommand{\bulym}{\raisebox{0.06em}{$\textcolor{Black}{\scriptstyle(y_m)}$}}

\newcommand{\BULXN}{\raisebox{0.06em}{$\textcolor{Black}{\scriptstyle(\tilde{x}_n)}$}}
\newcommand{\bulxn}{\raisebox{0.06em}{$\textcolor{Black}{\scriptstyle(x_n)}$}}

\newcommand{\BULXK}{\raisebox{0.06em}{$\textcolor{Black}{\scriptstyle(\tilde{x}_k)}$}}
\newcommand{\bulxk}{\raisebox{0.06em}{$\textcolor{Black}{\scriptstyle(x_k)}$}}

\newcommand{\BULZK}{\raisebox{0.06em}{$\textcolor{Black}{\scriptstyle(\tilde{z}_k)}$}}
\newcommand{\bulzk}{\raisebox{0.06em}{$\textcolor{Black}{\scriptstyle(z_k)}$}}

\newcommand{\DIATFP}{\raisebox{0.06em}{$\textcolor{Black}{\scriptstyle\langle\,\hat{t}_{F'}\,\rangle}$}}
\newcommand{\diatfp}{\raisebox{0.06em}{$\textcolor{Black}{\scriptstyle\langle\,{t}_{F'}\,\rangle}$}}

\newcommand{\BOXTFP}{\raisebox{0.06em}{$\textcolor{Black}{\scriptstyle\lbrack\,\check{t}_{F'}\,\rbrack}$}}
\newcommand{\boxtfp}{\raisebox{0.06em}{$\textcolor{Black}{\scriptstyle\lbrack\,{t}_{F'}\,\rbrack}$}}

\newcommand{\DIASF}{\raisebox{0.06em}{$\textcolor{Black}{\scriptstyle\langle\,\hat{s}_{F}\,\rangle}$}}

\newcommand{\DIATF}{\raisebox{0.06em}{$\textcolor{Black}{\scriptstyle\langle\,\hat{t}_{F}\,\rangle}$}}
\newcommand{\diatf}{\raisebox{0.06em}{$\textcolor{Black}{\scriptstyle\langle\,{t}_{F}\,\rangle}$}}

\newcommand{\BOXSF}{\raisebox{0.06em}{$\textcolor{Black}{\scriptstyle\lbrack\,\check{s}_{F}\,\rbrack}$}}

\newcommand{\BOXTF}{\raisebox{0.06em}{$\textcolor{Black}{\scriptstyle\lbrack\,\check{t}_{F}\,\rbrack}$}}
\newcommand{\boxtf}{\raisebox{0.06em}{$\textcolor{Black}{\scriptstyle\lbrack\,{t}_{F}\,\rbrack}$}}

\newcommand{\cgr}[1]{\raisebox{0.06em}{$\textcolor{Black}{\scriptstyle #1}$}}

\definecolor{Black}{rgb}{0.31, 0.31, 0.31}
\newcommand{\sgr}[1]{\raisebox{0.06em}{$\textcolor{Black}{\scriptstyle #1}$}}

\newcommand{\BULSF}{\raisebox{0.06em}{$\textcolor{Black}{\scriptstyle(\tilde{s}_F)}$}}
\newcommand{\bulsf}{\raisebox{0.06em}{$\textcolor{Black}{\scriptstyle(s_F)}$}}
\newcommand{\BULTF}{\raisebox{0.06em}{$\textcolor{Black}{\scriptstyle(\tilde{t}_F)}$}}
\newcommand{\bultf}{\raisebox{0.06em}{$\textcolor{Black}{\scriptstyle(t_F)}$}}

\newcommand{\BULRT}{\raisebox{0.06em}{$\textcolor{Black}{\scriptstyle(\tilde{r}_T)}$}}
\newcommand{\bulrt}{\raisebox{0.06em}{$\textcolor{Black}{\scriptstyle(r_T)}$}}

\newcommand{\bulst}{\raisebox{0.06em}{$\textcolor{Black}{\scriptstyle(s_T)}$}}

\newcommand{\BULSFP}{\raisebox{0.06em}{$\textcolor{Black}{\scriptstyle(\tilde{s}_{F'})}$}}

\newcommand{\BULTFP}{\raisebox{0.06em}{$\textcolor{Black}{\scriptstyle(\tilde{t}_{F'})}$}}
\newcommand{\bultfp}{\raisebox{0.06em}{$\textcolor{Black}{\scriptstyle(t_{F'})}$}}

\newcommand{\BULSX}{\raisebox{0.06em}{$\textcolor{Black}{\scriptstyle(\tilde{s}_x)}$}}

\newcommand{\bultx}{\raisebox{0.06em}{$\textcolor{Black}{\scriptstyle(t_x)}$}}

\newcommand{\bulstxt}{\raisebox{0.06em}{$\textcolor{Black}{\scriptstyle(s(t_{\overline{x}}/\overline{x})_T)}$}}
\newcommand{\BULSTXFP}{\raisebox{0.06em}{$\textcolor{Black}{\scriptstyle(\tilde{s(t_{\overline{x}}/\overline{x})}_{F'})}$}}

\newcommand{\BULRSXT}{\raisebox{0.06em}{$\textcolor{Black}{\scriptstyle(\tilde{r(s_{\overline{x}}/\overline{x})}_T)}$}}
\newcommand{\bulrsxt}{\raisebox{0.06em}{$\textcolor{Black}{\scriptstyle(r(s_{\overline{x}}/\overline{x})_T)}$}}
\newcommand{\BULSFY}{\raisebox{0.06em}{$\textcolor{Black}{\scriptstyle(\tilde{s}_{F\setminus\{y\}})}$}}
\newcommand{\bulsfy}{\raisebox{0.06em}{$\textcolor{Black}{\scriptstyle(s_{F \setminus \{y\} })}$}}
\newcommand{\BULSVZYX}{\raisebox{0.06em}{$\textcolor{Black}{\scriptstyle(\tilde{s_v, \overline{z}_{\overline{z}}, \overline{y}_{\overline{y}}, \overline{x}_{\overline{x}}})}$}}
\newcommand{\bulsvzyx}{\raisebox{0.06em}{$\textcolor{Black}{\scriptstyle(s_v, \overline{z}_{\overline{z}}, \overline{y}_{\overline{y}}, \overline{x}_{\overline{x}})}$}}
\newcommand{\BULTVZYX}{\raisebox{0.06em}{$\textcolor{Black}{\scriptstyle(\tilde{t_v, \overline{z}_{\overline{z}}, \overline{y}_{\overline{y}}, \overline{x}_{\overline{x}}})}$}}
\newcommand{\bultvzyx}{\raisebox{0.06em}{$\textcolor{Black}{\scriptstyle(t_v, \overline{z}_{\overline{z}}, \overline{y}_{\overline{y}}, \overline{x}_{\overline{x}})}$}}
\newcommand{\BULZYX}{\raisebox{0.06em}{$\textcolor{Black}{\scriptstyle(\tilde{\overline{z}_{\overline{z}}, \overline{y}_{\overline{y}}, \overline{x}_{\overline{x}}})}$}}

\newcommand{\BULTYRX}{\raisebox{0.06em}{$\textcolor{Black}{\scriptstyle(\tilde{t_y, r_{\overline{x}}})}$}}

\newcommand{\BULSYRX}{\raisebox{0.06em}{$\textcolor{Black}{\scriptstyle(\tilde{s_y, r_{\overline{x}}})}$}}

\newcommand{\BULZYTOFP}{\raisebox{0.06em}{$\textcolor{Black}{\scriptstyle(\tilde{z_y, t_{F'}})}$}}

\newcommand{\BULSYTOF}{\raisebox{0.06em}{$\textcolor{Black}{\scriptstyle(\tilde{s_y, t_F})}$}}

\newcommand{\BULTXYOY}{\raisebox{0.06em}{$\textcolor{Black}{\scriptstyle(\tilde{t_x, \overline{y}_{\overline{y}}})}$}}

\newcommand{\BOXTXYOY}{\raisebox{0.06em}{$\textcolor{Black}{\scriptstyle\lbrack\,\check{t_x, \overline{y}_{\overline{y}}}\,\rbrack}$}}

\newcommand{\BULYOYO}{\raisebox{0.06em}{$\textcolor{Black}{\scriptstyle(\tilde{\overline{y}_{\overline{y}}})}$}}
\newcommand{\bulyoyo}{\raisebox{0.06em}{$\textcolor{Black}{\scriptstyle(\overline{y}_{\overline{y}})}$}}

\newcommand{\DIAZYSOF}{\raisebox{0.06em}{$\textcolor{Black}{\scriptstyle\langle\,\hat{z_y, s_F}\,\rangle}$}}

\newcommand{\bulrxtf}{\raisebox{0.06em}{$\textcolor{Black}{\scriptstyle(r_x, t_F)}$}}

\newcommand{\bulsxtf}{\raisebox{0.06em}{$\textcolor{Black}{\scriptstyle(s_x, t_F)}$}}

\newcommand{\bulsytf}{\raisebox{0.06em}{$\textcolor{Black}{\scriptstyle(s_y, t_F)}$}}

\newcommand{\bulzytf}{\raisebox{0.06em}{$\textcolor{Black}{\scriptstyle(z_y, t_F)}$}}
\newcommand{\BULZYSF}{\raisebox{0.06em}{$\textcolor{Black}{\scriptstyle(\tilde{z_y, s_F})}$}}
\newcommand{\bulzysf}{\raisebox{0.06em}{$\textcolor{Black}{\scriptstyle(z_y, s_F)}$}}
\newcommand{\BULZYTFP}{\raisebox{0.06em}{$\textcolor{Black}{\scriptstyle(\tilde{z_y, t_{F'}})}$}}
\newcommand{\bulzytfp}{\raisebox{0.06em}{$\textcolor{Black}{\scriptstyle(z_y, t_{F'})}$}}
\newcommand{\RTFP}{\hat{R}({t_{F'}})}
\newcommand{\RSX}{\hat{R}({s_{\overline{x}}})}
\newcommand{\RTX}{\hat{R}({t_{\overline{x}}})}
\newcommand{\RTSXF}{\hat{R}(t(s_{\overline{x}}/\overline{x})_{F})}
\newcommand{\EST}{s \:\hat{=}\: t}
\newcommand{\ETT}{t \:\hat{=}\: t}
\newcommand{\ETS}{t \:\hat{=}\: s}
\newcommand{\ERTRF}{r_3 \:\hat{=}\: r_4}
\newcommand{\ERURT}{r_1 \:\hat{=}\: r_2}

\newcommand{\Q}[1]{\ensuremath{[\!\langle\texttt{Q}{#1}\rangle\!]}\xspace}

\newcommand{\IP}[1]{(\!({#1})\!)\xspace}
\newcommand{\ip}[1]{({#1})\xspace}
\newcommand{\ST}[1]{(\!({#1})\!)\xspace}
\newcommand{\st}[1]{({#1})\xspace}
\newcommand{\SA}[1]{[\!\langle{#1}\rangle\!]}

\newcommand{\Var}{\ensuremath{\mathsf{Var}}\xspace}

\newcommand{\FV}{\ensuremath{\mathsf{FV}}\xspace}

\newcommand{\xdashrightarrow}[2][]{\ext@arrow 0359\rightarrowfill@@{#1}{#2}}

\newcommand{\xdasharrow}[2][->]{
	\tikz[baseline=-\the\dimexpr\fontdimen22\textfont2\relax]{
		\node[anchor=south,font=\scriptsize, inner ysep=1.5pt,outer xsep=2.2pt](x){#2};
		\draw[shorten <=3.4pt,shorten >=3.4pt,dashed,#1](x.south west)--(x.south east);
	}
}

\theoremstyle{plain}
\newtheorem{thm}{Theorem}

\newtheorem{lem}[thm]{Lemma}
\newtheorem{cor}[thm]{Corollary}
\newtheorem{prop}[thm]{Proposition}

\newtheorem{lemma}[thm]{Lemma}

\theoremstyle{definition}

\newtheorem{definition}[thm]{Definition}

\newtheorem{claim}[thm]{Claim}
\newtheorem{remark}[thm]{Remark}
{\bfseries}{\itshape}

\renewenvironment{proof}[1]{\par\noindent\textit{Proof}.\space#1}{\hfill  $\square$}

\title{First order logic properly displayed}

\author[2]{Samuel Balco}
\author[1]{Giuseppe Greco}
\author[4]{Alexander Kurz}
\author[4]{Andrew Moshier}
\author[1,3]{Alessandra Palmigiano\thanks{The research of the second and fifth author has been funded in part by the NWO grant KIVI.2019.001.}}
\author[1]{Apostolos Tzimoulis}
\affil[1]{Vrije Universiteit Amsterdam}
\affil[2]{University of Leicester}
\affil[3]{University of Johannesburg}
\affil[4]{Chapman University} 

\date{}

\begin{document}

\maketitle

\begin{abstract}
We introduce a proper display calculus for first-order logic, of which we prove soundness, completeness, conservativity, subformula property and cut elimination via a Belnap-style metatheorem. All inference rules are closed under uniform substitution and are without side conditions.

\ 

{\em Math.~Subject Class.} 03B10, 03B35, 03B45, 03B47, 03F03, 03F05, 03F07, 03G05, 03G10, 03G15, 03G30, 06A06, 06A11, 06D10, 06D50, 06E15.
\end{abstract}


\section{Introduction}

In the proof-theoretic literature, the  treatment of quantifiers in first-order logic typically follows the lines of the original Gentzen's sequent calculi, and correspondingly, the introduction rules for first-order quantifiers\footnote{\textcolor{red}{We use the standard notation from \cite{Negri_van_Plato} where $A[t/x]$ signifies that the free occurrences of the variable $x$ in the formula $A$ are substituted with the term $t$ (see Section \ref{ssec:folprel} for the formal definition of substitutions). We also assume the usual eigenvariable conditions, namely $y$ must not occur free anywhere in the conclusion of the rules $\forall_R$ and $\exists_L$. Notice that the notation $A[y/x]$ in the premises of $\forall_R$ and $\exists_L$ is more restrictive and signifies that the formula has been derived for an arbitrary variable $y$ (not just a term $t$).}} are mostly additive: 
\begin{center}
	\begin{tabular}{rcl}
		\AX$A[t/x],\Gamma  \fCenter \Delta$
		\LeftLabel{$\forall_L$}
		\UI$\forall xA,\Gamma\fCenter \Delta$
		\DisplayProof
		&&
		\AX$\Gamma \fCenter A[y/x],\Delta$
		\RightLabel{$\forall_R$}
		\UI$\Gamma \fCenter\forall x A,\Delta$
		\DisplayProof
		\\
		&&\\
		
		\AX$A[y/x],\Gamma \fCenter \Delta$
		\LeftLabel{$\exists_L$}
		\UI$\exists x A,\Gamma \fCenter \Delta$
		\DisplayProof
		&&
		\AX$\Gamma \fCenter A[t/x],\Delta$
		\RightLabel{$\exists_R$}
		\UI$\Gamma \fCenter \exists x A,\Delta$
		\DisplayProof
		
		\\
	\end{tabular}
\end{center}

This is reflected also in the early display calculi literature, where quantifiers are not paired with structural connectives, up to Wansing \cite{Wansing1999,wansing1998displaying,ciabattoni2014hypersequent}, who introduces a display calculus for fragments of first-order logic based on the idea that quantifiers can be treated as modal operators in modal logic, and correspondingly defines introduction rules for quantifiers which involve their structural counterparts.

The key observation motivating this approach is that the existential quantifier behaves similarly to a modal diamond operator, and dually, the universal quantifier behaves similarly to a modal box operator. What underlies these similarities (observed and exploited in \cite{kuhn1980quantifiers,vbpredicate,montague,von1952double,Wansing1999,wansing1998displaying}), and semantically supports the requirement that the resulting calculus enjoys the display property (see e.g.~\cite{Belnap,wansing1998displaying,GJLPT-LE-logics} for the formal definition of this notion), is the order-theoretic notion of adjunction: indeed, the set-theoretic interpretations of the existential and universal quantifiers are the left and right adjoint  of the inverse projection map respectively, and more generally, in categorical semantics, the left and right adjoint of the pullbacks along projections \cite{lawvere1966functorial},\cite[Chapter 15]{goldblatt2014topoi}.

The display calculus of \cite{Wansing1999} contains rules with side conditions on the free and bound variables of formulas, similar to the ones of the original Gentzen sequent calculus. This implies that the introduction rules for first-order quantifiers are not closed under uniform substitution, i.e.\ the display calculus introduced in \cite{Wansing1999,wansing1998displaying} is not \emph{proper} \cite[Section 4.1]{wansing1998displaying}. 

In this paper we overcome this difficulty and introduce a {\em proper} display calculus for first-order logic. The main idea is that, as was the case of other logical frameworks (cf.~e.g.~\cite{Multitype,inquisitive,GP:linear}), a suitable \emph{multi-type} presentation makes it possible to encode the side conditions into analytic (structural) rules involving different types. Wansing's insight that quantifiers can be treated proof-theoretically as modal operators naturally embeds into the multi-type approach by simply regarding $(\forall x)$ and $(\exists x)$ as modal operators  bridging different types (i.e.~as \emph{heterogeneous} operators). Following Lawvere \cite{lawvere1966functorial,Lawvere:equality,Lawvere:adjointness} and Halmos \cite{halmos1954polyadic}, this requires to consider as many types as there are finite sets of free variables; that is, two first-order formulas have the same type if and only if they have exactly the same free variables.

Following these ideas, we introduce a proper multi-type display calculus for classical first-order logic, and show its soundness, completeness, conservativity, cut elimination and subformula property. The paper is organized as follows. In Section \ref{ssec:folprel}, we gather  preliminary notions, definitions and notation for first-order logic. In Section \ref{sec:semanalfol}, we recast the models of first-order logic in a framework amenable to support the semantics of the multi-type calculus for first-order logic, especially regarding the adjunction properties of quantification and substitution. From this semantic framework, we extract the defining conditions of the algebraic multi-type semantics for the calculus. In Section \ref{sec:mtfol}, we introduce the multi-type language of first-order logic. In Section \ref{sec:multicalc}, we introduce the display calculus for first-order classical logic. In Section \ref{sec:folproperties} we prove its soundness, completeness, conservativity, cut elimination and  subformula property. In Section \ref{sec:condir}, we summarize the results of this article and collect further research directions.

\section{Preliminaries on first-order logic}\label{ssec:folprel}

In this section we collect definitions and basic facts about first-order logic and introduce the notation that will be used throughout the paper. 

\paragraph{Language.} Let $\mathsf{Var}=\{v_1,\ldots,v_n,\ldots\}$ be a countable set of variables. Throughout the paper we let $\mathcal{P}_\omega(\mathsf{Var})$ denote the set of finite subsets of $\mathsf{Var}$. A first-order logic $\mathcal{L}$ over $\mathsf{Var}$ consists of a set of relation symbols $(R_i)_{i\in I}$ each of finite arity $n_i$; a set of function symbols $(f_j)_{j\in J}$ each of finite arity $n_j$; and a set of constant symbols $(c_k)_{k\in K}$ ($0$-ary functions). The language of first-order logic is built up from \emph{terms} defined recursively as follows:
\begin{align*}\mathsf{Trm}\ni t::= v_m\mid c_k\mid f_j(t,\ldots,t).\end{align*}
The formulas of first-order logic are defined recursively as follows:
\begin{align*}\mathcal{L}\ni A::=R_i(\overline{t})\mid t_1= t_2\mid\top\mid\bot\mid A\aand A\mid A\lor A\mid A\rightarrow A\mid \forall yA\mid \exists yA \end{align*}
\noindent where $\overline{t}$ is a sequence of terms of arity $n_i$ and negation $\neg A$ is defined as usual as $A \rightarrow \bot$. 

 For any term $t$ and formula $A$ we let $\mathsf{FV}(t)\in\mathcal{P}_\omega(\mathsf{Var})$ and $\mathsf{FV}(A)\in\mathcal{P}_\omega(\mathsf{Var})$ denote the sets of free variables of $t$ and $A$ respectively, recursively defined as follows:
\begin{center}
	\begin{tabular}{r c l }
		$\mathsf{FV}(v_m)$ &$ = $& $\{v_m\}$\\
		$\mathsf{FV}(c_k)$ &$ = $& $\varnothing$\\
		$\mathsf{FV}(f_j(t_1,\ldots,t_{n_j}))$& $= $& $\bigcup_{1\leq\ell\leq n_j}\mathsf{FV}(t_\ell)$ \\
		$\mathsf{FV}(\top)$ &$ = $& $\varnothing$\\
		$\mathsf{FV}(\bot)$ &$ = $& $\varnothing$\\
		$\mathsf{FV}(t_1=t_2)$ &$ = $& $\mathsf{FV}(t_1)\cup\mathsf{FV}(t_1)$\\
		$\mathsf{FV}(R_i(t_1,\ldots,t_{n_i}))$ &$ = $& $\bigcup_{1\leq\ell\leq n_i}\mathsf{FV}(t_\ell)$\\
		$\mathsf{FV}(A\aand B)$ &$ = $& $\mathsf{FV}(A)\cup \mathsf{FV}(B)$ \\
		$\mathsf{FV}(A\lor B)$ &$ = $& $\mathsf{FV}(A)\cup \mathsf{FV}(B)$ \\
		$\mathsf{FV}(A\rightarrow B)$ &$ = $& $\mathsf{FV}(A)\cup \mathsf{FV}(B)$ \\
		$\mathsf{FV}(\forall yA)$ &$ = $& $\mathsf{FV}(A)\setminus\{y\}$ \\
		$\mathsf{FV}(\exists yA)$ &$ = $& $\mathsf{FV}(A)\setminus\{y\}$. \\
	\end{tabular}
\end{center}
 In what follows, we will often identify sequences of variables and terms with the set containing the elements of the sequence. Hence, given a sequence of terms $\overline{t}$ we define $\mathsf{FV}(\overline{t})$ as the union of the sets of free variables of the elements of the sequence. Finally, if $\overline{t}$ is a sequence of terms and $s$ is a term, we let  $s,\overline{t}$ denote the sequence obtained by prefixing $s$ to $\overline{t}$. We will also abuse notation and write e.g.~$\overline{x}\supseteq \mathsf{FV}(s)$ to signify that every free variable in $s$ occurs in the sequence $\overline{x}$. In what follows we will conflate notation and use $\mathcal{L}$ to denote signature, formulas and set of theorems.

\paragraph{Substitution.} In our treatment of first-order logic we assume that substitution happens simultaneously for  all free variables of a term or a formula. If $\overline{x}\supseteq\mathsf{FV}(s)$ and $\overline{x}\supseteq\mathsf{FV}(A)$ and for each $v_m\in\overline{x}$ we let  $t_{v_m}$ denote the corresponding term which will be substituted for $v_m$, then $t_{\overline{x}}$ denotes a sequence of terms over $\overline{x}$\footnote{\label{footnote:substitution}That is,  $t_{\overline{x}}\in \mathsf{Trm}^{\overline{x}}$ and $t_{\overline{x}}(v_m)=t_{v_m}$.}, $(t_{\overline{x}}/\overline{x})$ denotes the simultaneous substitution of each $v_m$ with $t_{v_m}$ and we define $s(t_{\overline{x}}/\overline{x})$ and $A(t_{\overline{x}}/\overline{x})$ recursively as follows:
\begin{center}
	\begin{tabular}{r c l }
		$v_n(t_{\overline{x}}/\overline{x})$ &$ = $&
		$t_{v_n}$\\
		$c_k(t_{\overline{x}}/\overline{x})$ &$ = $&
		$c_k$\\
		$f(s_1,\ldots,s_n)(t_{\overline{x}}/\overline{x})$& $= $& $f(s_1(t_{\overline{x}}/\overline{x}),\ldots,s_n(t_{\overline{x}}/\overline{x}))$ \\
		$\top(t_{\overline{x}}/\overline{x})$ &$ = $&
		$\top$\\
		$\bot(t_{\overline{x}}/\overline{x})$ &$ = $&
		$\bot$\\
		$(s_1=s_2)(t_{\overline{x}}/\overline{x})$ &$ = $& $s_1(t_{\overline{x}}/\overline{x})=s_2(t_{\overline{x}}/\overline{x})$\\
		$R(s_1,\ldots,s_n)(t_{\overline{x}}/\overline{x})$ &$ = $& $R(s_1(t_{\overline{x}}/\overline{x}),\ldots,s_n(t_{\overline{x}}/\overline{x}))$\\
		$(A\aand B)(t_{\overline{x}}/\overline{x})$ &$ = $& $A(t_{\overline{x}}/\overline{x}) \aand B(t_{\overline{x}}/\overline{x})$ \\
		$(A\lor B)(t_{\overline{x}}/\overline{x})$ &$ = $& $A(t_{\overline{x}}/\overline{x}) \lor B(t_{\overline{x}}/\overline{x})$ \\
		$(A\rightarrow B)(t_{\overline{x}}/\overline{x})$ &$ = $& $A(t_{\overline{x}}/\overline{x}) \rightarrow B(t_{\overline{x}}/\overline{x})$ \\
		$(\forall yA)(t_{\overline{x}}/\overline{x})$ &$ = $& $\forall z(A(t_{\overline{x}}/\overline{x})),\quad z\notin\mathsf{FV}(t_{\overline{x}})\cup\mathsf{FV}(\forall yA)$ \\
		$(\exists yA)(t_{\overline{x}}/\overline{x})$ &$ = $& $\exists z(A(t_{\overline{x}}/\overline{x})),\quad z\notin\mathsf{FV}(t_{\overline{x}})\cup\mathsf{FV}(\exists yA)$ \\
	\end{tabular}
\end{center}

\paragraph{Models.}The models of a first-order logic $\mathcal{L}$  are tuples $$M=(D,(R^D_i)_{i\in I},(f^D_j)_{j\in J},(c^D_k)_{k\in K})$$ where $D$ is a non empty set (the domain of the model) and  $R^D_i,f^D_j,c^D_k$ are concrete $n_i$-ary relations over $D$, $n_j$-ary functions on $D$ and elements of $D$ respectively interpreting the symbols of the language in the model $M$. The interpretation of a formula in $M$ is facilitated by variable interpretation maps $\nu:\mathsf{Var}\to D$ and is given recursively as follows:
{\small \begin{center}
	\begin{tabular}{l c l }
		$(v_m)^D$ & $=$ & $\nu(v_m)$\\
		$(f_j(t_1,\ldots,t_{n_j}))^D$ & $=$ & $(f^D_j(t^D_1,\ldots,t^D_{n_j}))$\\
		$M,\nu\models \top$ & & Always\\
		$M,\nu\models \bot$ & & Never\\
		$M,\nu\models R_i(t_{\overline{x}})$ &$ \iff$& $R^D_i(t^D_{\overline{x}}) $\\
		$M,\nu\models t_1=t_2$ &$ \iff$& $t^D_1=t^D_2 $\\
		$M,\nu\models A\aand B$ &$ \iff$& $M,\nu\models A\text{ and }M,\nu\models B$\\
		$M,\nu\models A\lor B$ &$ \iff$& $M,\nu\models A\text{ or }M,\nu\models B$\\
		$M,\nu\models A\rightarrow B$ &$ \iff$& $M,\nu\models A\text{ implies }M,\nu\models B$\\
		$M,\nu\models \forall yA$ &$ \iff$& $M,\nu'\models A$ for all $\nu'$ such that $\nu'(x)=\nu(x)$ for all $x\neq y$\\	
		$M,\nu\models \exists yA$ &$ \iff$& $M,\nu'\models A$ for some $\nu'$ such that $\nu'(x)=\nu(x)$ for all $x\neq y$.\\	
\end{tabular}
\end{center}}

\paragraph{Axiomatic system.}The following deductive system \cite{Enderton01}, denoted with $\vdash_{\mathrm{FO}}$, is sound and complete w.r.t.\ the models mentioned above:
\begin{enumerate}
	\item Propositional tautologies;
	\item $\forall x (B\rightarrow C)\rightarrow (\forall xB\rightarrow\forall xC)$;
	\item $B\rightarrow\forall xB$, where $x\notin\mathsf{FV}(B)$;
	\item $\forall xB\rightarrow B(t/x)$;
	\item Equality axioms:
		\begin{enumerate}
			\item $t=t$;
			\item $s=t\to  (A(s/x)\to A(t/x))$;
		\end{enumerate}
	\item All universal closures of instances of the above;
	\item Modus ponens.
\end{enumerate}

\section{Semantic analysis}\label{sec:semanalfol}

\textcolor{red}{In the present section we use the same notational conventions we will adopt in the calculus D.FO and D.FO$^\ast$. The letter symbolizing a function could occur in round, angular or square brackets. Angular (alluding to diamonds) for left adjoints, Round for functions that are simultaneously both left and right adjoints, and square (alluding to boxes) for right adjoints denote functions here and not function symbols.}
	

In order to develop a multi-type environment in which the insights discussed in the introduction can be carried out, we need as many types as there are finite sets of variables. Therefore, types $F$ correspond to elements of $\mathcal{P}_\omega(\mathsf{Var})$. On the semantic side, each such type interprets formulas $A$ such that $\mathsf{FV}(A)=F$, and hence, being closed under all propositional connectives, each type is naturally endowed with the structure of a Boolean algebra. Thus, given a first-order logic $\mathcal{L}$ over a countable set of variables $\mathsf{Var}$, and letting $\mathsf{Var}_x:=\mathsf{Var}\setminus\{x\}$ for every $x\in\mathsf{Var}$, the semantic environment that supports the multi-type presentation of $\mathcal{L}$ is based on structures
$\mathbb{H}=(\mathcal{A},\mathcal{Q},\mathcal{S})$,  where
\begin{itemize}
	\item $\mathcal{A}= \{\mathbb{A}_{F}\mid  F\in\mathcal{P}_{\omega}(\mathsf{Var})\}$;
	\item $\mathcal{Q}= \{{\bulx_{T}}, {\boxx _{F}}, {\diax _{F}} \mid x\in F\in \mathcal{P}_{\omega}(\mathsf{Var})\text{ and }
	T\in \mathcal{P}_{\omega}(\mathsf{Var}_x)\}$;
	\item $\mathcal{S}=\{\bultf  , \boxtf,\diatf \mid F\in\mathcal{P}_{\omega}(\mathsf{Var})\text{ and } t_F\in\mathsf{Trm}^F\}$ (see Footnote \ref{footnote:substitution}),
\end{itemize}
where every $\mathbb{A}_{F}$ is a Boolean algebra, and the following adjunctions hold for every finite set of free variables $F\in\mathcal{P}_{\omega}(\mathsf{Var})$, every $x\in F$, and every $T\in\mathcal{P}_{\omega}(\mathsf{Var}_x)$:
$$
\xymatrix{
	\mathbb{A}_{T\cup\{x\}}
	\ar@{<-}[rr]|-{\bulx _{T}}
	\ar@/^2pc/[rr]^{\boxx _{T\cup\{x\}}}
	\ar@/_2pc/[rr]_{\diax _{T\cup\{x\}}}
	\ar@{{}{ }{}} @/^1pc/[rr]|{\top}
	\ar@{{}{ }{}} @/_1pc/[rr]|{\top}
	& & \mathbb{A}_{T}
}\quad\quad
\xymatrix{
	\mathbb{A}_{F'}
	\ar@{<-}[rr]|-{\bultf  }
	\ar@/^2pc/[rr]^{\boxtf}
	\ar@/_2pc/[rr]_{\diatf}
	\ar@{{}{ }{}} @/^1pc/[rr]|{\top}
	\ar@{{}{ }{}} @/_1pc/[rr]|{\top}
	& & \mathbb{A}_{F}
}
$$
Notice that in $\bulx_{T}, \boxx _{F}, \diax _{F}$ the subscripts denote the domain of the functions. In practice we will always omit the subscript since it will be obvious from the context.

The adjunctions illustrated above justify the soundness of the following rules (see Section \ref{sec:multicalc} for the definition of the language and the multi-type display calculi D.FO and D.FO$^\ast$):

\begin{center}
	\begin{tabular}{cc}
		\AX$\DIAX X \fCenter_{\!\!F\setminus\{x\}\,\,} Y$
		\doubleLine
		\UI$X \fCenter_{\!\!F\cup\{x\}\,\,} \BULX Y$
		\DisplayProof
		&
		\AX$Y \fCenter_{\!\!F\setminus\{x\}\,\,} \BOXX X$
		\doubleLine
		\UI$\BULX Y \fCenter_{\!\!F\cup\{x\}\,\,} X$
		\DisplayProof \\
		
	\end{tabular}
\end{center}
\noindent where $\DIAX$ (resp.~$\BOXX$) is the structural connective corresponding to existential (resp.~universal) quantification, and $\BULX $ is the structural connective corresponding to the inverse projection map;

\begin{center}
	\begin{tabular}{c}
		\bottomAlignProof
		\AX$\BULTF X \fCenter_{\!\!F'\,\,} Y$
		\doubleLine
		\UI$X\fCenter_{\!\!F\,\,} \BOXTF Y$
		\DisplayProof
	\end{tabular}
	\begin{tabular}{c}
		
		\bottomAlignProof
		\AX$Y \fCenter_{\!\!F'\,\,} \BULTF X$
		\doubleLine
		\UI$\DIATF Y \fCenter_{\!\!F\,\,} X$
		\DisplayProof
	\end{tabular}
\end{center}
\noindent where $\BULTF $ is the structural connective corresponding to the uniform substitution of $F$ with $t_F$, and $\BOXTF$ (resp.~$\DIATF$) is its right (resp.~left) adjoint. 

In order to identify the relevant properties which we need to impose on the multi-type environment outlined above, in what follows, we establish a systematic connection between the standard $\mathcal{L}$-models and a subclass of the structures described above.


\medskip\noindent 

Let $M$ be a model of some fixed but arbitrary  first-order language $\mathcal{L}$.
Let $S,T$ be finite subsets of variables 
and let $M^S$ and $M^T$ denote the sets of functions into $M$ with domains $S$ and $T$ respectively. If \mbox{$f:M^S\to M^T$} is a map, then $\texttt{graph}(f)$ induces a complete Boolean algebra homomorphism\footnote{In general, every relation $R\subseteq X\times Y$ induces maps $\langle R\rangle,[R]:\mathcal{P}(Y)\to\mathcal{P}(X)$ respectively defined by the assignments $\langle R\rangle A= R^{-1}[A]=\{x\in X\mid \exists y( R(x,y)\ \&\ y\in A)\}$ and $[R] A= (R^{-1}[A^c])^c=\{x\in X\mid \forall y(R(x,y)\Rightarrow y\in A)\}$, which have a right adjoint $[R^{-1}]$ and a left adjoint $\langle R^{-1}\rangle$ respectively. When $R=\text{graph}\gbulf $ for some function $f:X\to Y$ then $\langle R\rangle=[R]$ and we denote it $\gbulf $, and abbreviate $\langle R^{-1}\rangle$ as $\gdiaf $ and $[R^{-1}]$ as $\gboxf $.} \mbox{$\gbulf:\mathcal{P}(M^T)\to\mathcal{P}(M^S)$}, between the Boolean algebras $\mathcal{P}(M^T)$ and $\mathcal{P}(M^S)$ of the potential interpretations of $T$\emph{-predicates}\footnote{By $T$-predicate we mean a predicate $P$ such that $\mathsf{FV}(P)=T$.} and the potential interpretations of $S$-predicates, defined as follows: for any $B\subseteq M^T$,
$$\gbulf (B) =f^{-1}[B]= \{ m\in M^S \mid f(m)\in B\}.$$ Hence $\gbulf$ has both a right-adjoint $\gboxf$ and a left-adjoint $\gdiaf$, as in the following picture: $$
\xymatrix{
	2^{M^S}
	\ar@{<-}[rr]|-{\gbulf}
	\ar@/^2pc/[rr]^{\gboxf}
	\ar@/_2pc/[rr]_{\gdiaf}
	\ar@{{}{ }{}} @/^1pc/[rr]|{\top}
	\ar@{{}{ }{}} @/_1pc/[rr]|{\top}
	& & 2^{M^T}
}
$$

which are defined by the following assignments: for  any $A\subseteq M^S$,
\begin{align*}
\gboxf (A) & = \{n\in M^T \mid \forall m\in M^S\,.\, f(m) = n \Rightarrow m\in A\} \\
\gdiaf (A)  & = \{n\in M^T \mid \exists m\in M^S\,.\, f(m) = n \ \& \ m\in A\}.
\end{align*}
Notice in particular that if $A=\{m\}\in J^\infty(\mathcal{P}(M^S))$ for some $m: S\to M$, then $\gdiaf (\{m\}) = \{f(m)\}$.

\bigskip
\noindent 
There are two types of relevant instances of maps $f:M^S\to M^T$. The first is given for $S:=T\cup\{x\}$ for some  $x\notin T$, and $f:=\pi_x:M^{T\cup \{x\}}\to M^T$ is the canonical projection along $T$. In this case, $\gbulpx (A)$ is the \emph{cylindrification} (cf.~\cite{henkin1971cylindric}) of $A$ over the $x$-coordinate, and the adjoint maps $\gboxpx$, $\gdiapx$ are the semantic counterparts of the usual first-order logic quantifiers $\forall x$ and $\exists x$. The second type of instance is given for $f$ arising from a simultaneous substitution $(t/\overline{x})$. In this case, $S:=\mathsf{FV}(t_{\overline{x}})$, $T:=\overline{x}=\{x_1,\ldots,x_n\}$, and $f:=t_{\overline{x}}:M^S\to M^T$ is defined by the assignment: for any $m\in M^S$, 
$$t_{\overline{x}}(m)=(t_{x_1}(m),\ldots,t_{x_n}(m)).$$ 
In this case, the corresponding $\gbultx$ is such that, if $B\subseteq M^T$ is the semantic interpretation of a $T$-predicate $P$, then $\gbultx (B)$ is the semantic interpretation of the $S$-predicate $(t_{\overline{x}}/\overline{x})P$ resulting from applying the simultaneous substitution $(t_{\overline{x}}/\overline{x})$ to $P$. That is, in this case, $\gbultx$ is the semantic interpretation of the substitution from which $t_{\overline{x}}$ arises.

\begin{prop}\label{prop:somefolproperties}
	For every model $M$ for $\mathcal{L}$ and for every $f:M^S\to M^T$ the following inclusions hold:\footnote{The connective $\pdra$ denotes the left residual of join and, in the case of Boolean algebras, $A \pdra B = \neg A \wedge B$.}
	\begin{center}
		\begin{tabular}{c c c}
			$\gboxf (A)\cup \gboxf (B) \ \subseteq \ \gboxf (A\cup B)$ & & $\gdiaf (A\cap B) \ \subseteq \  \gdiaf (A)\cap \gdiaf (B)$ \\
			$\gdiaf (A)\Rightarrow \gboxf (B)\ \subseteq \ \gboxf (A\Rightarrow B)$ & & $\gdiaf (A\pdra B)\ \subseteq \ \gboxf (A)\pdra\gdiaf (B)$\\
		\end{tabular}
	\end{center}
	
\end{prop}
\begin{proof}
	The inclusions in the first row are a consequence of the monotonicity of $\gboxf $ and $\gdiaf $.
	As to proving the left hand inclusion of the second row, by adjunction and residuation it is enough to show that
	$$A\cap \gbulf (\gdiaf (A)\Rightarrow \gboxf (B))\subseteq B. $$
	Since $A\subseteq \gbulf (\gdiaf (A))$ it is enough to show that
	$$\gbulf (\gdiaf (A))\cap \gbulf (\gdiaf (A)\Rightarrow \gboxf (B))\subseteq B. $$
	Since $\gbulf $ is meet-preserving this is equivalent to
	$$\gbulf (\gdiaf (A)\cap (\gdiaf (A)\Rightarrow \gboxf (B)))\subseteq B  $$
	which, by adjunction, is equivalent to
	$$\gdiaf (A)\cap (\gdiaf (A)\Rightarrow \gboxf (B))\subseteq \gboxf (B) $$ which holds by the definition of implication (also when the implication is defined intuitionistically). The proof of the right hand side inclusion of the second row is dual and uses the join-preservation of $\gbulf $.\end{proof}

Summing up, we are justified in introducing the following

\begin{definition}\label{def:hetmod}
	For every model $M$ for $\mathcal{L}$, the  {\em heterogeneous powerset $\mathcal{L}$-model} associated with $M$ is the tuple $(\mathbb{H}_{M},V)$ such that $\mathbb{H}_{M}=(\mathcal{A}_M,\mathcal{Q}_M,\mathcal{S}_M)$ is defined as follows\footnote{In category-theoretic parlance, this can be seen as a product-preserving functor from the category of terms to the category of Boolean algebras.}:
	\begin{itemize}
		\item $\mathcal{A}_M= \{\mathcal{P}(M^F) \mid  F\in\mathcal{P}_{\omega}(\mathsf{Var})\}$;
		\item $\mathcal{Q}_M=\{\gbulpx, \gboxpx, \gdiapx  , \mid F\in\mathcal{P}_{\omega}(\mathsf{Var}_x)\text{ and }
		\pi_x:M^{F\cup\{x\}}\to M^{F}\}$;
		\item $\mathcal{S}_M=\{\bultf  , \boxtf,\diatf \mid F\in\mathcal{P}_{\omega}(\mathsf{Var})\text{ and } t_F\in\mathsf{Trm}^F\}$,
	\end{itemize}
	and $V$ is such that $V(A)\in\mathcal{P}(M^F)$ for every atomic formula $A$ of $\mathcal{L}$ with $\mathsf{FV}(A)=F$. \textcolor{red}{In Section \ref{sec:mtfol} we formally introduce the multi-type propositional language $\mathcal{L}_{\mathrm{MT}}$ which is naturally interpreted on  heterogeneous models such as those above. Notice that, in the case of heterogeneous powerset $\mathcal{L}$-models, $\gbulpx  ,\gboxpx  ,\gdiapx  $ are the concrete functions arising from $\pi_x$ as described above.} 
\end{definition}

\begin{prop}\label{prop:mainsoundness}
	For every $\mathcal{L}$-model $M$, the heterogeneous $\mathcal{L}$-model associated with $M$ satisfies the following identities:

\begin{center}
\begin{tabular}{@{}r@{}c@{}lr@{}c@{}lrl}
$\gbulf (A\cap B)$                     & $\ =\ $ & $\gbulf (A)\cap\gbulf (B)$                    & $\gbulf (A\cup B)$       & $\ =\ $ &$\gbulf (A)\cup\gbulf (B)$ & (a) \\
$\gbulf (A\Rightarrow B)$         & $=$ & $\gbulf (A)\Rightarrow\gbulf (B)$             & $\gbulf (A\pdra B)$     & $=$ &$\gbulf (A)\pdra\gbulf (B)$    & (b) \\
$\gbulpx \gboxpy (A)$               & $=$ & $\gboxpy \gbulpx (A)$                             & $\gbulpx \gdiapy (A)$ & $=$ &$\gdiapy \gbulpx (A)$            & (c) \\
$\gbultf \gboxpx (A)$   & $=$ & $\gboxpz \gbulzytf(A)$                        & $\gbultf \gdiapy  (A)$  & $=$ & $\gdiapz \gbulzytf (A)$         &(d)  \\
$\gbulpx \gbulpy (A)$               & $=$ & $\gbulpy \gbulpx (A)$                              & $\gbultf \gbulst (A)$     & $=$ & $\gbulstxt (A)$                     & (e) \\
$\gbulsytf \gbulpy (A)$ & $=$ & $\cgr{(\pi_{z_1})}\cdots\cgr{(\pi_{z_k})}\sgr{(t_F)}(A),$ &                                     &        &                                              & (f) \\
\end{tabular}
\end{center}

\noindent where $x\neq y$, $z\notin\mathsf{FV}(t_F)\cup F$ and $\{z_1,\ldots,z_k\}=\mathsf{FV}(s)\setminus\mathsf{FV}(t_F)$ and $f\in\{\pi_x,t_F\}$. 
\end{prop}
\begin{proof}
	The identities in \emph{(a)} and \emph{(b)} immediately follow from $\gbulpx$ being a Boolean algebra homomorphism. Notice that the remaining identities are all Sahlqvist, so, in what follows, we will show that their correspondents computed with ALBA \cite{CoPa:non-dist} hold in the model.
	
	As to the identities in \emph{(c)}, they both have the same ALBA-reduct, namely $$\gdiapx \gbulpy \mathbf{j}=\gbulpy \gdiapx \nomj,$$ where $\nomj\in J^\infty(\mathcal{P}(M^{F\cup\{x\}}))$, that is, $\nomj$ can be identified with a map $\nomj:F\cup\{x\}\to M$. As discussed above, $\gbulpy $ takes any such $\nomj$ in the appropriate type to the $y$-cylindrification of $\nomj$, and $\gdiapx$ forgets the $x$-coordinate, then it is clear that  these two operations commute when $x\neq y$.
	
	The identities in \emph{(d)} have the following ALBA-reducts respectively:
	$$\gbulzytf \gbulpz \nomj = \gbulpy \gdiatf \nomj\quad\text{and}\quad
	\gbultf \gdiapy \nomj= \gdiapz \gbulzytf \nomj, $$ where $\nomj\in J^\infty(\mathcal{P}(M^{\mathsf{FV}(t_F)}))$, that is, we can write $\nomj:\mathsf{FV}(t_F)\to M$. Notice that $\gbulpz$ takes $\nomj$ to its $z$-cylindrification in $\mathcal{P}(M^{\mathsf{FV}(t_F)\cup\{z\}})$, and $\gdiazytf$ transforms any element $\nomj'$ of this cylindrification  by sending any coordinate different from $z$ to the corresponding $t(\nomj')$ and renaming the element in the $z$-coordinate by declaring it the $y$-coordinate; moreover $\gdiatf$ transforms $\nomj$ into $t(\nomj)_F$, and $\gbulpy $ cylindrifies it by adding the $y$-coordinate. It is clear that these two compositions yield the same outcome. The second identity is argued analogously.
	
	The left-hand identity in \emph{(e)} has the following ALBA-reduct: $$\gdiapx \gdiapy \, \nomj = \gdiapy \gdiapx \, \nomj.$$ Recalling that, with a bit of notational abuse, $\gdiaf (\nomj)=f(\nomj)$, the assumption that $x\neq y$ guarantees that the order in which the projections are applied does not matter.
	
	The right-hand identity in \emph{(e)} has the following ALBA-reduct:
	$$\gdiastxt \, \nomj = \gdiast \gdiatf \, \nomj.$$ This identity expresses that applying $t_F$ to $\nomj$ and then $s_T$ to $t(\nomj)_F$ is the same as applying the compounded substitution map $s(t_{\overline{x}}/\overline{x})_T$ to $\nomj$, which is trivially true.
	
	Finally, the identity \emph{(f)} has the following ALBA-reduct:
$$\gdiatf \gdiapzu \cdots \gdiapzn \, \nomj = \gdiapy \gdiasytf \, \nomj.$$ 
This identity expresses that applying $t_F$ after having forgotten the coordinates $z_1,\ldots,z_n$ is the same as applying $t_F$ in parallel with any substitution $s_y$ with $\{z_1,\ldots,z_n\}=\mathsf{FV}(s)\setminus\mathsf{FV}(t_F)$ and $y\notin F$, and then forgetting the $y$-coordinate. This is again immediately true.
\end{proof}

\begin{remark}
	Notice that while the identities in \emph{(a)} and \emph{(b)} in the proposition above hold for any function $f$, the remaining identities hold true because of the particular functions involved, and hence by choosing different functions, corresponding to different notions of quantifiers and substitutions, these properties might change.
\end{remark}

Proposition \ref{prop:mainsoundness} provides the guidelines for completing the definition of heterogeneous $\mathcal{L}$-algebras. Namely, we are going to use the identities \emph{(a)-(f)} as parts of the following

\begin{definition}\label{def:heterogeneousalgebras}
	For any first-order logic $\mathcal{L}$ over a denumerable set of individual variables $\mathsf{Var}$, a \emph{heterogeneous  $\mathcal{L}$-algebra} is a tuple
	$\mathbb{H}=(\mathcal{A},\mathcal{Q},\mathcal{S})$, such that
\begin{itemize}
	\item $\mathcal{A}= \{\mathbb{A}_{F} \mid F \in \mathcal{P}_{\omega} (\mathsf{Var})\}$;
	\item $\mathcal{Q}= \{\bulxt, \boxxf, \diaxf, \mid x \in F \in \mathcal{P}_{\omega} (\mathsf{Var})\text{ and }
	T\in \mathcal{P}_{\omega}(\mathsf{Var}_x)\}$;
	\item $\mathcal{S}=\{\bultf, \boxtf, \diatf \mid F\in\mathcal{P}_{\omega}(\mathsf{Var})\text{ and } t_F\in\mathsf{Trm}^F\}$,
\end{itemize}
\noindent where for every $F\in\mathcal{P}_{\omega}(\mathsf{Var})$, $\mathbb{A}_{F}$ is a Boolean algebra and $$
	
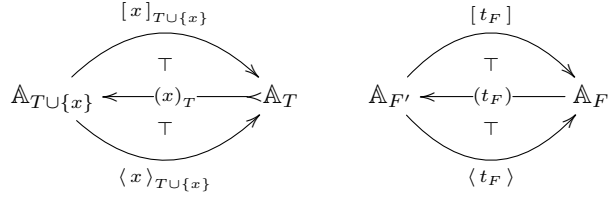
\xymatrix{
		\mathbb{A}_{T\cup\{x\}}
		\ar@{<-<}[rr]|-{\bulx_{T}}
		\ar@/^2pc/[rr]^{\boxx _{T\cup\{x\}}}
		\ar@/_2pc/[rr]_{\diax _{T\cup\{x\}}}
		\ar@{{}{ }{}} @/^1pc/[rr]|{\top}
		\ar@{{}{ }{}} @/_1pc/[rr]|{\top}
		& & \;\mathbb{A}_{T}
	}\quad\quad
	\xymatrix{
		\mathbb{A}_{F'}
		\ar@{<-}[rr]|-{\bultf  }
		\ar@/^2pc/[rr]^{\boxtf}
		\ar@/_2pc/[rr]_{\diatf}
		\ar@{{}{ }{}} @/^1pc/[rr]|{\top}
		\ar@{{}{ }{}} @/_1pc/[rr]|{\top}
		& & \mathbb{A}_{F}
	}
	$$ 
such that $\bulxt$ is an order embedding and  the following equations hold (we omit mentioning the types):

\begin{center}
\begin{tabular}{rclrclr}
$\bulstar (a\aand b)$          & $=$ & $\bulstar (a)\aand\bulstar (b)$                    & $\bulstar (a\lor b)$    & $=$ & $\bulstar (a)\lor\bulstar (b)$    &{\em (a)}  \\
$\bulstar (a\rightarrow b)$ & $=$ & $\bulstar (a)\rightarrow\bulstar (b)$           & $\bulstar (a\pdra b)$ & $=$ & $\bulstar (a)\pdra\bulstar (b)$ & {\em (b)} \\
$\bulx \boxy (a)$                & $=$ &$\boxy \bulx (a)$                                        & $\bulx \diay (a)$       & $=$ & $\diay \bulx (a)$                      &{\em (c)}   \\
$\bultf \boxy (a)$                & $=$ & $\boxz \bulzytf (a)$                                   & $\bultf \diay (a)$        & $=$ & $\diaz \bulzytf (a)$                  & {\em (d)}  \\
$\bulx \buly (a)$                 & $=$ & $\buly \bulx (a)$                                        & $\bultf \bulst (a)$      & $=$ & $\bulstxt (a)$                           & {\em (e)} \\
$\bulsytf \buly (a)$             & $=$ & $ \bulzu \cdots\bulzk \sgr{(t_F)}(a),$ &                                  &        &                                                 &{\em (f)}   \\
\end{tabular}
\end{center}

\noindent where $x \neq y$, $z \notin \mathsf{FV}(t_F) \cup F$, $\bulstar \in \{\bulx, \bultf\}$ and $\{z_1, \ldots, z_k\} = \mathsf{FV}(s)\setminus\mathsf{FV}(t_F)$. A heterogeneous $\mathcal{L}$-algebra is \emph{perfect} if every $\mathbb{A}_F$ is.\footnote{A Boolean algebra is {\em perfect} if it is complete and completely join-generated by its atoms. Equivalently, perfect Boolean algebras are isomorphic to powerset algebras.}
\end{definition}

Clearly, even the class of perfect heterogeneous $\mathcal{L}$-algebras is much larger than the class of those arising from $\mathcal{L}$-models. However, as we will discuss in the next section, there is a way in which $\mathcal{L}$ is sound with respect to this larger class of models.
\section{Multi-type presentation of first-order logic}\label{sec:mtfol}
Let $\mathcal{L}$ be a first-order language over a denumerable set of individual variables $\mathsf{Var}$. The notion  of heterogeneous $\mathcal{L}$-algebras (cf.\ Definition \ref{def:heterogeneousalgebras}) naturally comes with the following \emph{multi-type propositional language} $\mathcal{L}_{\mathrm{MT}}$ canonically interpreted in it. Types in this language bijectively correspond to elements  $F\in \mathcal{P}_\omega(\mathsf{Var})$. The sets $\mathcal{L}_F$ of $F$-formulas, for each such $F$, are defined by simultaneous induction as follows:

\begin{center}
\begin{tabular}{rcl}
$\mathcal{L}_F\ni A^F$ & $::=$ & $R(t_{\overline{x}}) \mid t=s \mid$ \\
                                      &          & $\rule{0pt}{2.30ex}\top^F \mid \bot^F\mid A^F\aand A^F\mid A^F\lor A^F\mid A^F\rightarrow A^F \mid A^F\pdra A^F \mid$ \\
                                      &          & $\rule{0pt}{2.30ex}\boxy A^{F\cup\{y\}} \mid \diay A^{F\cup\{y\}} \mid \bulx A^{F\setminus\{x\}} \mid \bultfp A^{F'}$ \\
\end{tabular}
\end{center}

\noindent where $\mathsf{FV}(t_{F'})=F$, $y\notin F$, $x\in F$, and $R$ is any relation symbol of $\mathcal{L}$, and $t_{F'}: F'\to \mathsf{Trm}$. The symbol $\bultfp$ denotes the simultaneous substitution of $t_{v_m}$ for $v_m$ for every $v_m\in F'$.

\begin{definition}\label{def:hetmodel}
	A \emph{heterogeneous algebraic $\mathcal{L}$-model} is a tuple $(\mathbb{H}, V)$ such that $\mathbb{H}$ is a heterogeneous $\mathcal{L}$-algebra and $V$ maps atomic propositions $R(t_{F'})$ and $t=s$ in $\mathcal{L}_F$  to elements of $\mathbb{A}_F$ so that for any terms $r$, $s$ and $t_{\overline{x}}$ and any relation symbol $P$ (also including equality and $\top$, $\bot$):
\begin{center}
		\begin{tabular}{c c}
		1.$V(P(t_{\overline{x}}))=\bultx V(P(\overline{x}))$ & 2.$V(r=r)=\top$ \\
		\multicolumn{2}{c}{3. $\bulxu \cdots \bulxn V(r=s) \aand \bulyu \cdots \bulym V(P(r,t_{\overline{x}})) \leq \bulzu \cdots \bulzk V(P(s,t_{\overline{x}}))$ } 
	\end{tabular}
\end{center}
where $\{x_1,\ldots,x_n\}=\FV(t_F)\setminus\FV(r,s)$, $\{y_1,\ldots,y_m\}=\FV(s)\setminus\FV(r,t_F)$ and $\{z_1,\ldots,z_k\}=\FV(r)\setminus\FV(s,t_F)$.
\end{definition}

The definition of the interpretation $(\mathbb{H}, V)\models A^F$ of $\mathcal{L}_{\mathrm{MT}}$-formulas  into heterogeneous $\mathcal{L}$-algebras straightforwardly generalizes the definition of the interpretation of propositional languages in algebras of compatible signature.

The discussion of the previous section also motivates the definition of the following translation:

\begin{definition}\label{def:trremsub}
Define the translation $(\cdot)^\tau: \mathcal{L}_{\mathrm{MT}}\to \mathcal{L}$ as follows:
	
	\begin{center}
		\begin{tabular}{r c l c r c l}
			$R(t_{F'})^\tau$ &$ = $& $R(t_{F'})$ & $\quad\quad$ & $t=s^\tau$ & $ = $& $t=s$\\
			$\top^\tau$ &$ = $& $\top$ && $\bot^\tau$ &$ = $& $\bot$\\
			$(A\aand B)^\tau$ &$ = $& $A^\tau \aand B^\tau$ && $(A\lor B)^\tau$ &$ = $& $A^\tau \lor B^\tau$\\
			$(A\rightarrow B)^\tau$ &$ = $& $A^\tau \rightarrow B^\tau$ && $(A\pdra B)^\tau$ &$ = $& $A^\tau \pdra B^\tau$\\
			$(\boxy A)^\tau$ &$ = $& $ \forall yA^\tau$ && $(\diay A)^\tau$ &$ = $& $\exists yA^\tau$\\
			$(\bulx A)^\tau$ &$ = $& $ A^\tau$ && $(\bultfp A)^\tau$ &$ = $& $A^\tau(t_{F'}/\overline{x})$\\
		\end{tabular}
	\end{center}
	
\end{definition} 
\begin{lem}\label{lem:subrules}
Let $(\mathbb{H},V)$ be a heterogeneous algebraic $\mathcal{L}$-model. Let $P$ be a relation symbol of $\mathcal{L}$ or the equality. If 
$$(\sgr{(t_{F_1})} \cdots \sgr{(t_{F_k})} P(t_{\overline{x}}))^\tau = (\sgr{(s_{T_1})} \cdots \sgr{(s_{T_m})} P(s_{\overline{x}}))^\tau$$
 then 
$$\sgr{(t_{F_1})} \cdots \sgr{(t_{F_k})} V(P(t_{\overline{x}}))= \sgr{(s_{T_1})} \cdots \sgr{(s_{T_m})} V(R(s_{\overline{x}})).$$ 
\end{lem}
\begin{proof}
Let
	\begin{equation}\label{eq:axcond}
	(\sgr{(t_{F_1})} \cdots \sgr{(t_{F_k})} R(t_{\overline{x}}))^\tau = (\sgr{(s_{T_1})} \cdots \sgr{(s_{T_m})} R(s_{\overline{x}}))^\tau.
	\end{equation}
Let us assume that $(\sgr{(t_{F_1})} \cdots \sgr{(t_{F_k})} R(t_{\overline{x}}))^\tau=R(r_{\overline{x}})$, that is 
$$(\sgr{(t_{F_1})} \cdots \sgr{(t_{F_k})} \sgr{(t_{\overline{x}})} R(\overline{x}))^\tau=R(r_{\overline{x}})$$ 
therefore the composition of $\sgr{(t_{F_1})} \cdots \sgr{(t_{F_k})} \sgr{(t_{\overline{x}})}$ is $\sgr{(r_{\overline{x}})}$. By applying the right-hand equality of \emph{(e)} of Definition \ref{def:heterogeneousalgebras} a finite number of times we obtain \begin{equation}\label{eq:half1}
	\sgr{(t_{F_1})} \cdots \sgr{(t_{F_k})} \sgr{(t_{\overline{x}})} V(R(\overline{x}))= \sgr{(r_{\overline{x}})} V(R(\overline{x}))=V(R(r_{\overline{x}}))
	\end{equation}  the second equality holding by item 1 of  Definition \ref{def:hetmodel}.  By assumption \eqref{eq:axcond},
	$$ (\sgr{(s_{T_1})} \cdots \sgr{(s_{T_m})}  R(s_{\overline{x}}))^\tau=R(r_{\overline{x}}).$$ Similar reasoning as above yields
	\begin{equation}\label{eq:half2}
	\sgr{(s_{T_1})} \cdots \sgr{(s_{T_m})} \sgr{(s_{\overline{x}})} V(R(\overline{x}))=V(R(r_{\overline{x}})).
	\end{equation}  By \eqref{eq:half1} and \eqref{eq:half2} we get  $$\sgr{(t_{F_1})} \cdots \sgr{(t_{F_k})} \sgr{(t_{\overline{x}})} V(R(\overline{x}))= 	\sgr{(s_{T_1})} \cdots \sgr{(s_{T_m})} \sgr{(s_{\overline{x}})} V(R(\overline{x})).$$ Since $\sgr{(t_{\overline{x}})} V(R(\overline{x}))=V(R(t_{\overline{x}}))$ and $\sgr{(s_{\overline{x}})} V(R(\overline{x}))=V(R(s_{\overline{x}}))$ by Definition \ref{def:hetmodel}, $$\sgr{(t_{F_1})} \cdots \sgr{(t_{F_k})} V(R(t_{\overline{x}}))= 	\sgr{(s_{T_1})} \cdots \sgr{(s_{T_m})}   V(R(s_{\overline{x}})),$$ which completes the proof. \end{proof}

\begin{lem}\label{lem:eqgen}
	Let $A$ be a formula with free variable $x$. Then 
	$$\bulxu \cdots \bulxn V(r=s) \aand \bulyu \cdots \bulym \sgr{(r_x,t_F)}V(A) \leq \bulzu \cdots \bulzk \sgr{(s_x,t_F)}V(A),$$
	where $\{x_1,\ldots,x_n\}=\FV(t_F)\setminus\FV(r,s)$, $\{y_1,\ldots,y_m\}=\FV(s)\setminus\FV(r,t_F)$ and $\{z_1,\ldots,z_k\}=\FV(r)\setminus\FV(s,t_F)$.
\end{lem}
\begin{proof}
	We proceed by induction on the complexity of $A$. The base case follows immediately from 1 and 3 of Definition \ref{def:hetmodel}.  Since $\bulx  $ and $\bulsf   $ are Boolean algebra homomorphisms the cases when $A$ is $B\aand C$ and $\bot$  follow easily. The case when $A$ is $B\to C$ uses the fact that $\bulx  $ and $\bulsf   $ are Boolean algebra homomorphisms and utilizes the following reasoning: Assume that $x\aand b_1\leq c_1$ and $x\aand c_2\leq b_2$. We have: 
\begin{align*}
	& x\aand b_1\leq c_1\\
	\Rightarrow & x\aand b_2 \aand (b_2\to b_1)\leq c_1\\
	\Rightarrow & x\aand c_2 \aand (b_2\to b_1)\leq c_1\\
	\Rightarrow &x\aand \aand (b_2\to b_1)\leq c_2\aand c_1.
\end{align*}  

If $A$ is $\buly    B$, $\diay     B$, $\boxy B$ we use the equalities (c)-(f) of Definition \ref{def:heterogeneousalgebras}\footnote{Notice that this proof works also for Heyting algebras.}.
\end{proof}

\begin{cor}\label{lem:eqrules}
	Let $A$ be a formula with free variable $x$. If
	$$\bulxu \cdots \bulxn V(r=s)\leq\bulyu \cdots\bulym \bulrxtf V(A)$$
	 then 
	 $$\bulxu \cdots \bulxn V(r=s) \leq \bulzu \cdots \bulzk \bulsxtf V(A)$$
	where $\{x_1,\ldots,x_n\}=\FV(t_F)\setminus\FV(r,s)$, $\{y_1,\ldots,y_m\}=\FV(s)\setminus\FV(r,t_F)$ and $\{z_1,\ldots,z_k\}=\FV(r)\setminus\FV(s,t_F)$.
\end{cor}
\begin{proof}
	Immediately from Lemma \ref{lem:eqgen}.\end{proof}
\begin{prop}\label{prop:sematicfaith}
	For every $\mathcal{L}$-model $M$ and every $A\in\mathcal{L}_{\mathrm{MT}}$, $$M\models A^\tau\quad\text{iff}\quad (\mathbb{H}_M,V)\models A.$$
\end{prop}
\begin{proof}
	By straightforward induction on $A$.
\end{proof}

\noindent The identities \emph{(a)-(f)} of Proposition \ref{prop:mainsoundness} have the following syntactic counterparts in the language $\mathcal{L}_{\mathrm{MT}}$:

\begin{center}
\begin{tabular}{rclrclrc}
$\bulstar (A\aand B)$          & $\dashv\vdash$ & $\bulstar A\aand\bulstar B$           & $\bulstar (A\lor B)$   & $\dashv\vdash$ & $\bulstar A\lor\bulstar B$    & (a) \\
$\bulstar (A\rightarrow B)$ & $\dashv\vdash$ & $\bulstar A\rightarrow\bulstar B$ & $\bulstar (A\pdla B)$ & $\dashv\vdash$ & $\bulstar A\pdla\bulstar B$ & (b) \\
$\bulx \boxy A$                  & $\dashv\vdash$ & $\boxy \bulx A$                             & $\bulx \diay A$          & $\dashv\vdash$ & $\diay \bulx A$                    & (c) \\
$\bultf \boxy A$                  & $\dashv\vdash$ & $\boxz \bulzytf A$                         & $\bultf \diay A$          & $\dashv\vdash$ & $\diaz \bulzytf A$                & (d) \\
$\bulx \buly A$                   & $\dashv\vdash$ & $\buly \bulx A$                              & $\bultf \bulst A$         &$\dashv\vdash$  & $\bulstxt A$                         & (e) \\
$\bulsytf \buly A$               & $\dashv\vdash$ & $ \bulzu \cdots\bulzk \bultf  A.$     &                                   &                           &                                             & (f) \\
\end{tabular}
\end{center}
These inequalities are all analytic inductive (cf.\ \cite[Definition 55]{GMPTZ}), and hence give rise to analytic structural rules. We will exploit this fact in the next section when introducing the rules of the calculus.



\section{Multi-type calculus for first-order logic}\label{sec:multicalc}
For any first-order language $\mathcal{L}$ (cf.\ Section \ref{ssec:folprel}) the structural and operational connectives of the proper multi-type display calculi $\mathrm{D.FO}$ and $\mathrm{D.FO}^*$ are the following:
\begin{itemize}
	\item Logical and structural homogeneous connectives for any type $F$:
\end{itemize}
\begin{center}
\begin{tabular}{|r|c|c|c|c|}
\hline
\scriptsize{Structural symbols} & $\rule{0pt}{2.5ex}\AATOP$ & $\ABOT$ & $\RTX$ & $\EST$ \\
\hline
\scriptsize{Operational symbols} & $\top$ & $\bot$ & $R(t_{\overline{x}})$ & $s = t$ \\
\hline
\end{tabular}
\end{center}
\begin{center}
\begin{tabular}{|r|c|c|c|c|c|c|}
\hline
\scriptsize{Structural symbols}  & $\rule{0pt}{2.25ex}\AAND$ & $\AOR$ & $\ADRARR$ & $\ADLARR$ & $\ARARR$ & $\ALARR$ \\
\hline
\scriptsize{Operational symbols} & $\aand$ & $\lor$ & \!$(\adrarr)$ & $(\adlarr)$ & $\ararr$ & $(\leftarrow)$\\
\hline
\end{tabular}
\end{center}
\begin{itemize}
\item Logical and structural heterogeneous for each $x,y\in \Var$:
\end{itemize}
\begin{center}
\begin{tabular}{|r|c|c|c|c|c|c|}
\hline &
\fns $\mathcal{L}_{F\setminus\{x\}} \to \mathcal{L}_{F}$ & \mc{2}{c|}{\fns $\mathcal{L}_{F\cup\{y\}} \to \mathcal{L}_{F}$} & \fns $\mathcal{L}_{F'} \to \mathcal{L}_{F}$ & \mc{2}{c|}{\fns $\mathcal{L}_{F} \to \mathcal{L}_{F'}$}\\
\hline
\scriptsize{Structural symbols} &
$\BULX$ & $\DIAY$ & $\BOXY$  & $\BULTFP$ & $\DIATFP$ & $\BOXTFP$ \\
\hline
\scriptsize{Operational symbols} &
$\bulx$   & \ $\diay$\ \ & $\boxy $ & $\bultfp$     & $(\diatfp) $ & $ (\boxtfp) $ \\
\hline
\end{tabular}
\end{center}

In the table above, some operational symbols above  ($\adrarr$, $\adlarr$, $\leftarrow$, $\diatfp$ and $\boxtfp$) appear within brackets to signify that, unlike their associated structural symbols, they occur only in the language and calculus for $\mathrm{D.FO}^*$.

Below we define the display calculi  $\mathrm{D.FO}^*$ and $\mathrm{D.FO}$. Notice that formally each type, $F$, has its own turnstile denoted by $\vdash_F$. Below to avoid overloading the notation we sometimes omit the subscript, since the type can be directly deduced by the free variables of the structures on the left and right of the turnstile. 

\begin{definition}
	\label{def:FODL}
	The display calculi $\mathrm{D.FO}^*$ and $\mathrm{D.FO}$ consist of the following display postulates, structural rules, and operational
	rules:
	\begin{enumerate}
		\item Identity and Cut:
		
		\begin{center}
			\begin{tabular}{rl}
				
				\AXC{$\phantom{X \fCenter_{\!\!F\,\,} A}$}
				\LL{\fns Id}
				\UI$\RTFP \fCenter R(t_{F'})$
				\DP
				 & 
				
				\AXC{$\phantom{X \fCenter_{\!\!F\,\,} A}$}
				\LL{\fns Id}
				\UI$\EST\fCenter s=t$
				\DP
			\end{tabular}
			\end{center}
		\begin{center}
			\begin{tabular}{c}
				\AX$X \fCenter_{\!\!F\,\,} A$
				\AX$A \fCenter_{\!\!F\,\,} Y$
				\RL{\fns Cut}
				\BI$X \fCenter_{\!\!F\,\,} Y$
				\DP
				
			\end{tabular}
		\end{center}
		
		\item Display postulates for homogeneous connectives: 
		\begin{center}
			\begin{tabular}{rl}
				\AX$X \AAND Y \fCenter_{\!\!F\,\,} Z$
				\doubleLine
				\UI$Y \fCenter_{\!\!F\,\,} X \ARARR Z$
				\DisplayProof
				&
				\AX$Z \fCenter_{\!\!F\,\,} X \AOR Y$
				\doubleLine
				\UI$X \ADRARR Z \fCenter_{\!\!F\,\,} Y$
				\DisplayProof
				\\ &\\
				\AX$X \AAND Y \fCenter_{\!\!F\,\,} Z$
				\doubleLine
				\UI$X \fCenter_{\!\!F\,\,} Z \ALARR Y$
				\DisplayProof
				&
				\AX$Z \fCenter_{\!\!F\,\,} X \AOR Y$
				\doubleLine
				\UI$Z \ADLARR Y \fCenter_{\!\!F\,\,} X$
				\DisplayProof
			\end{tabular}
		\end{center}
		\item Display postulates for quantifiers:
		\begin{center}
			\begin{tabular}{rl}
				\AX$X \fCenter_{\!\!F\cup\{x\}\,\,} \BULX Y$
				\doubleLine
				\UI$\DIAX X \fCenter_{\!\!F\setminus\{x\}\,\,} Y$
				\DisplayProof
				&
				\AX$\BULX Y \fCenter_{\!\!F\cup\{x\}\,\,} X$
				\doubleLine
				\UI$Y \fCenter_{\!\!F\setminus\{x\}\,\,} \BOXX X$
				\DisplayProof \\
			\end{tabular}
		\end{center}
		\item Display postulates for substitutions:
		\begin{center}
			\begin{tabular}{rl}
				\AX$Y \fCenter_{\!\!F'\,\,} \BULTF X$
				\doubleLine
				\UI$\DIATF Y \fCenter_{\!\!F\,\,} X$
				\DP
				&
				\AX$\BULTF X \fCenter_{\!\!F'\,\,} Y$
				\doubleLine
				\UI$X \fCenter_{\!\!F\,\,} \BOXTF Y$
				\DP
			\end{tabular}
		\end{center}
		
		\item Additional adjunction-related rules
		\begin{center}
			\begin{tabular}{rl}
				\AX$ \BULX \DIAX X \fCenter_{\!\!F\,\,} Y$
				\LeftLabel{\fns cadj}
				\UI$X \fCenter_{\!\!F\,\,}Y$
				\DP
				&
				\AX$\BULTFP \DIATFP X \fCenter_{\!\!F\,\,} Y$
				\LeftLabel{\fns sadj}
				\UI$X \fCenter_{\!\!F\,\,} Y$
				\DP
				 \\
			\end{tabular}
		\end{center}
		
		
		\item Necessitation for quantifiers:
		\begin{center}
			\begin{tabular}{rl}
				\AX$\AATOP_{\!F\cup\{x\}} \fCenter_{\!\!F\cup\{x\}\,\,} X$
				\doubleLine
				\UI$\BULX \AATOP_{\!F\setminus\{x\}} \fCenter_{\!\!F\cup\{x\}\,\,} X$
				\DP
				&
				\AX$X \fCenter_{\!\!F\cup\{x\}\,\,} \ABOT_{F\cup\{x\}}$
				\doubleLine
				\UI$X \fCenter_{\!\!F\cup\{x\}\,\,} \BULX \ABOT_{F\setminus\{x\}}$
				\DP
				\\
			\end{tabular}
		\end{center}
		
		\item Necessitation for substitution where $\mathsf{FV}(\overline{t}_{F'})=F$:
		\begin{center}
			\begin{tabular}{rl}
				\AX$\AATOP_{\!F} \fCenter_{\!\!F\,\,} X$
				\doubleLine
				\UI$\BULTFP \AATOP_{\!F'} \fCenter_{\!\!F\,\,} X$
				\DP
				&
				\AX$X \fCenter_{\!\!F\,\,} \ABOT_F$
				\doubleLine
				\UI$X \fCenter_{\!\!F\,\,} \BULTFP \ABOT_{F'}$
				\DP
				\\
			\end{tabular}
		\end{center}

		\item Structural rules for atomic formulas:
		
		\begin{center}
				\begin{tabular}{rl}
				\AX$\sgr{(\tilde{s}_{T_1})} \cdots \sgr{(\tilde{s}_{T_m})} \RSX \fCenter X$
				\UI$\sgr{(\tilde{t}_{F_1})} \cdots \sgr{(\tilde{t}_{F_k})} \RTX \fCenter X$
				\DP
				&
				\AX$\sgr{(\tilde{s}_{T_1})} \cdots \sgr{(\tilde{s}_{T_m})} \ERTRF \fCenter X$
				\UI$\sgr{(\tilde{t}_{F_1})} \cdots \sgr{(\tilde{t}_{F_k})} \ERURT \fCenter X$
				\DisplayProof
				\\
			\end{tabular}
		\end{center}
		
		where the sequents are well-typed and\footnote{We will discuss these axioms in the conclusions.}  $$(\sgr{(t_{F_1})} \cdots \sgr{(t_{F_k})} R(\overline{t}_{\overline{x}}))^\tau=(\sgr{(s_{T_1})} \cdots \sgr{(s_{T_m})} R(\overline{s}_{\overline{x}}))^\tau $$ and
		$$(\sgr{(t_{F_1})} \cdots \sgr{(t_{F_k})} r_1=r_2)^\tau=(\sgr{(s_{T_1})} \cdots \sgr{(s_{T_m})} r_3=r_4)^\tau .$$
		\item Structural rules encoding the behaviour of conjunction and disjunction:
	\begin{center}
			\begin{tabular}{rl}
				\AX$X \fCenter_{\!\!F\,\,} Y$
				\doubleLine
				\LeftLabel{\fns$\AATOP_{L}$}
				\UI$\AATOP \AAND X \fCenter_{\!\!F\,\,} Y$
				\DP
				&
				\AX$Y \fCenter_{\!\!F\,\,} X$
				\doubleLine
				\RightLabel{\fns$\ABOT_{R}$}
				\UI$Y \fCenter_{\!\!F\,\,} X \AOR \ABOT$
				\DP
				 \\ & \\
				\AX$Y \AAND X \fCenter_{\!\!F\,\,} Z$
				\LeftLabel{\fns$E_L$}
				\UI$X \AAND Y \fCenter_{\!\!F\,\,} Z $
				\DP
				&
				\AX$Z \fCenter_{\!\!F\,\,} X \AOR Y$
				\RightLabel{\fns$E_R$}
				\UI$Z \fCenter_{\!\!F\,\,} Y \AOR X$
				\DP
				\\
				&  \\
				\AX$Y \fCenter_{\!\!F\,\,} Z$
				\LeftLabel{\fns$W_L$}
				\UI$X \AAND Y \fCenter_{\!\!F\,\,} Z$
				\DP
				&
				\AX$Z \fCenter_{\!\!F\,\,} Y$
				\RightLabel{\fns$W_R$}
				\UI$Z \fCenter_{\!\!F\,\,} Y \AOR X$
				\DP
				\\ & \\
				\AX$X \AAND X \fCenter_{\!\!F\,\,} Y$
				\LeftLabel{\fns$C_L$}
				\UI$X \fCenter_{\!\!F\,\,} Y $
				\DP
				&
				\AX$Y \fCenter_{\!\!F\,\,} X \AOR X$
				\RightLabel{\fns$C_R$}
				\UI$Y \fCenter_{\!\!F\,\,} X$
				\DP
				\\
				&\\
			
					\AX$X \AAND (Y \AAND Z) \fCenter_{\!\!F\,\,} W$
					\doubleLine
					\LeftLabel{\fns$A_{L}$}
					\UI$(X \AAND Y) \AAND Z \fCenter_{\!\!F\,\,} W $
					\DP
				&
			
					\AX$W \fCenter_{\!\!F\,\,} (Z \AOR Y) \AOR X$
					\doubleLine
					\RightLabel{\fns$A_{R}$}
					\UI$W \fCenter_{\!\!F\,\,} Z \AOR (Y \AOR X)$
					\DP
			\end{tabular}
		\end{center}
	\item Structural rules for equality:
		\begin{center}
		\begin{tabular}{rl}
			\bottomAlignProof
			\AX$\ETT \AAND X \fCenter Y$
			\UI$X \fCenter Y$
			\DP
			&
			\AX$\BULXU \cdots \BULXN \ETS \fCenter \BULYU \cdots \BULYM \BULTYRX X$
			\UI$\BULXU \cdots\BULXN \ETS \fCenter \BULZU \cdots \BULZK \BULSYRX X$
			\DP
			\\
		\end{tabular}
	\end{center}
	where $\{x_1,\ldots,x_n\}=\FV(r_F)\setminus\FV(t,s)$, $\{y_1,\ldots,y_m\}=\FV(s)\setminus\FV(t,r_F)$ and $\{z_1,\ldots,z_k\}=\FV(t)\setminus\FV(s,r_F)$.
		\item Introduction rules for atomic formulas:
			\begin{center}
			\begin{tabular}{rl}
				\AX$\RSX \fCenter X$
				\UI$R(s_{\overline{x}}) \fCenter X$
				\DP
				&
				\AX$\EST \fCenter X$
				\UI$s=t \fCenter X$
				\DP
				\\
			\end{tabular}
		\end{center}

		\item Introduction rules for the propositional connectives:
		\begin{center}
				\begin{tabular}{rl}
					\AXC{\phantom{$\gbot \fCenter \I$}}
					\LL{\fns$\abot_L$}
					\UI$\abot \fCenter \ABOT$
					\DP
					&
					\AX$X \fCenter \ABOT$
					\RL{\fns$\abot_R$}
					\UI$X \fCenter \abot$
					\DP
					\\ &\\
					\AX$\AATOP \fCenter X$
					\LeftLabel{\fns$\aatop_L$}
					\UI$\aatop \fCenter X$
					\DP
					&
					\AXC{\phantom{$\I \fCenter \top$}}
					\RightLabel{\fns$\atop_R$}
					\UI$\AATOP \fCenter \aatop$
					\DP
					\\
					&  \\
					\AX$A \AAND B \fCenter X$
					\LeftLabel{\fns$\aand_L$}
					\UI$A \aand B \fCenter X$
					\DP
					&
					\AX$X \fCenter A$
					\AX$Y \fCenter B$
					\RightLabel{\fns$\aand_R$}
					\BI$X \AAND Y \fCenter A \aand B$
					\DP
				\\ &\\
					\AX$A \fCenter X$
					\AX$B \fCenter Y$
					\LeftLabel{\fns$\lor_L$}
					\BI$A \aor B \fCenter X \AOR Y$
					\DP
					&
					\AX$X \fCenter A \AOR B$
					\RightLabel{\fns$\lor_R$}
					\UI$X \fCenter A \lor B$
					\DP
					\\
					& \\
					\AX$X \fCenter A$
					\AX$B \fCenter Y$
					\LeftLabel{\fns$\to_L$}
					\BI$A \ararr B \fCenter X \ARARR Y$
					\DP
					&
					\AX$X \fCenter A \ARARR B$
					\RightLabel{\fns$\to_R$}
					\UI$X \fCenter A \to B$
					\DP
			\end{tabular}	
		\end{center}
		
		We omit the type since these rules involve only one type. In the presence of the exchange rules $E_L$ and $E_R$, the structural connective $<$ and the corresponding operational connectives $\pdla$ and $\pla$ are redundant.
		\item Grishin rules for classical logic:
		\begin{center}
			\begin{center}
				\begin{tabular}{rl}
					\AX$X \ADRARR (Y \AAND Z) \fCenter  W$
					\LL{\fns$Gri$}
					\UI$(X \ADRARR Y) \AAND Z\fCenter W$
					\DP
					&
					\AX$W\fCenter X \ARARR (Y \AOR Z)$
					\RL{\fns$Gri$}
					\UI$W\fCenter (X \ARARR Y) \AOR Z$
					\DP
					\\
				\end{tabular}
			\end{center}
		\end{center}
		
		\item Introduction rules for the heterogeneous connectives of $\mathrm{D.FO}$:
		\begin{center}
			\begin{tabular}{rl}
				\AX$\DIAX A\fCenter_{\!\!F\,\,} X$
				\LL{\fns$\diax _L$}
				\UI$\diax A \fCenter_{\!\!F\,\,} X$
				\DP
				&
				\AX$X \fCenter_{\!\!F\,\,} A$
				\RL{\fns$\diax _R$}
				\UI$\DIAX X \fCenter_{\!\!F\setminus\{x\}\,\,} \diax  A$
				\DP
				\\
				& \\
				\AX$A \fCenter_{\!\!F\,\,} X$
				\LL{\fns$\boxx _L$}
				\UI$\boxx A \fCenter_{\!\!F\setminus\{x\}\,\,} \BOXX A$
				\DisplayProof
				&
				\AX$X \fCenter_{\!\!F\,\,} \BOXX A$
				\RL{\fns$\boxx _R$}
				\UI$X \fCenter_{\!\!F\,\,} \boxx  A$
				\DP
				\\
			\end{tabular}
		\end{center}

		\begin{center}
			\begin{tabular}{rl}
				\AX$\BULX A \fCenter_{\!\!F\,\,} X$
				\LeftLabel{\fns$\bulx  _L$}
				\UI$\bulx A \fCenter_{\!\!F\,\,} X$
				\DP
				&
				\AX$X \fCenter_{\!\!F\,\,} \BULX  A$
				\RL{\fns$\bulx  _R$}
				\UI$X \fCenter_{\!\!F\,\,} \bulx   A$
				\DP
				\\
			\end{tabular}
		\end{center}	
		\begin{center}
			\begin{tabular}{rl}
				\AX$\BULTFP A \fCenter_{\!\!F\,\,} X$
				\LeftLabel{\fns$\bultfp_L$}
				\UI$\bultfp A \fCenter_{\!\!F\,\,} X$
				\DP
				&
				\AX$X \fCenter_{\!\!F\,\,} \BULTFP A$
				\RL{\fns$\bultfp_R$}
				\UI$X \fCenter_{\!\!F\,\,} \bultfp A$
				\DP
				\\
			\end{tabular}
		\end{center}	
		\item Introduction rules for the heterogeneous connectives of $\mathrm{D.FO}^*$:
		\begin{center}
			\begin{tabular}{rl}
				\AX$\DIATFP A \fCenter_{\!\!F'\,\,} X$
				\LL{\fns$\diatfp_L$}
				\UI$\diatfp A \fCenter_{\!\!F'\,\,} X$
				\DP
				&
				\AX$X \fCenter_{\!\!F\,\,} A$
				\RL{\fns$\diatfp_R$}
				\UI$\DIATFP X \fCenter_{\!\!F'\,\,} \diatfp  A$
				\DP
				\\
				& \\
				\AX$A \fCenter_{\!\!F\,\,} X$
				\LL{\fns$\boxtfp_L$}
				\UI$\boxtfp  A \fCenter_{\!\!F'\,\,} \BOXTFP A$
				\DP
				&
				\AX$X \fCenter_{\!\!F'\,\,} \BOXTFP A$
				\RL{\fns$\boxtfp_R$}
				\UI$X \fCenter_{\!\!F'\,\,} \boxtfp A$
				\DP
				\\
			\end{tabular}
		\end{center}

		\item Monotonicity and order embedding rules:
		\begin{center}
			\begin{tabular}{l r}
				\AX$X \fCenter_{\!\!F\setminus\{x\}\,\,} Y$
				\LL{\fns$\BULX_L$}
				\doubleLine
				\UI$\BULX X \fCenter_{\!\!F\cup\{x\}\,\,} \BULX Y$
				\DP
				&
				\AX$X \fCenter_{\!\!F'\,\,} Y$
				\LL{\fns$\BULTFP_R$}
				\UI$\BULTFP X \fCenter_{\!\!F\,\,} \BULTFP Y$
				\DP
			\end{tabular}
		\end{center}

		\item Interaction between homogeneous and heterogeneous connectives:
		\begin{center}	
				\begin{tabular}{rl}
					\AX$\BULX X \ADRARR \BULX Y \fCenter Z$
					\LL{\fns$(\BULX, \ADRARR)_L$}
					\doubleLine
					\UI$\BULX (X \ADRARR Y) \fCenter Z$
					\DP
					&
					\AX$Z \fCenter \BULX X \ARARR \BULX Y$
					\RL{\fns$(\BULX, \ARARR)_R$}
					\doubleLine
					\UI$Z \fCenter \BULX (X \ARARR Y)$
					\DP
					\\ & \\
					\AX$\BULX X \ADLARR \BULX Y \fCenter Z$
					\LL{\fns$(\BULX, \ADLARR)_L$}
					\doubleLine
					\UI$\BULX (X \ADLARR Y) \fCenter Z$
					\DP
					& 		
					\AX$Z \fCenter \BULX X \ALARR \BULX Y$
					\RL{\fns$(\BULX, \ALARR)_R$}
					\doubleLine
					\UI$Z \fCenter \BULX (X \ADLARR Y)$
					\DP
					\\
					& \\
					\AX$\BULX X \AAND \BULX Y \fCenter Z$
					\LL{\fns$(\BULX, \AAND)_L$}
					\doubleLine
					\UI$\BULX (X \AAND Y) \fCenter Z$
					\DP
					&
					\AX$Z \fCenter \BULX X \AOR \BULX Y$
					\RL{\fns$(\BULX, \AOR)_R$}
					\doubleLine
					\UI$Z \fCenter \BULX (X \AOR Y)$
					\DP
					\\ & \\
					\AX$\BULTF X \AAND \BULTF Y \fCenter Z$
					\LL{\fns$(\BULTF, \AAND)_L$}
					\doubleLine
					\UI$\BULTF (X \AAND Y) \fCenter Z$
					\DP
					&
					\AX$Z \fCenter \BULTF X \AOR \BULTF Y$
					\RL{\fns$(\BULTF, \AOR)_R$}
					\doubleLine
					\UI$Z \fCenter \BULTF (X \AOR Y)$
					\DP
					\\
					& \\
					\AX$\BULTF X \ADRARR \BULTF Y \fCenter Z$
					\LL{\fns$(\BULTF, \ADRARR)_L$}
					\doubleLine
					\UI$\BULTF (X \ADRARR Y) \fCenter Z$
					\DP
					&
					\AX$Z \fCenter \BULTF X \ARARR \BULTF Y$
					\RL{\fns$(\BULTF, \ARARR)_R$}
					\doubleLine
					\UI$Z \fCenter \BULTF (X \ARARR Y)$
					\DP
					\\ & \\
					\AX$\BULTF X \ADLARR \BULTF Y \fCenter Z$
					\LL{\fns$(\BULTF, \ADLARR)_L$}
					\doubleLine
					\UI$\BULTF (X \ADLARR Y) \fCenter Z$
					\DP
					&
					\AX$Z \fCenter \BULTF X \ALARR \BULTF Y$
					\RL{\fns$(\BULTF, \ALARR)_R$}
					\doubleLine
					\UI$Z \fCenter \BULTF (X \ALARR Y)$
					\DP
				\end{tabular}
				\end{center}

		\item Interaction between heterogeneous connectives: In what follows, \mbox{$x\neq y$,} $z\notin\mathsf{FV}(t_{F'})\cup F'$ and $\{z_1,\ldots,z_k\}=\mathsf{FV}(s)\setminus\mathsf{FV}(t_F)$:
	
	\begin{center}	
		
				\begin{tabular}{rl}
					\AX$\BULX \DIAY X \fCenter Y$
					\LL{\fns cq$_L$}
					\doubleLine
					\UI$\DIAY \BULX X \fCenter Y$
					\DisplayProof
					&
					\AX$Y \fCenter \BULX \BOXY X$
					\RL{\fns cq$_R$}
					\doubleLine
					\UI$Y \fCenter \BOXY \BULX X$
					\DP
					\\ & \\
					\AX$\BULTFP \DIAY X \fCenter Y$
					\LL{\fns sq$_L$}
					\doubleLine
					\UI$\DIAZ \BULZYTOFP X \fCenter Y$
					\DP
					 &
					\AX$Y \fCenter \BULTFP \BOXY X$
					\RL{\fns sq$_R$}
					\doubleLine
					\UI$Y \fCenter \BOXZ \BULZYTOFP X$
					\DP
					\\
					 & \\
					\AX$\BULX \BULY X \fCenter Y$
					\LL{\fns cc$_L$}
					\doubleLine
					\UI$\BULY \BULX X \fCenter Y$
					\DP
					 & 
					\AX$Y \fCenter \BULX \BULY X$
					\RL{\fns cc$_R$}
					\doubleLine
					\UI$Y \fCenter \BULY \BULX X$
					\DP
					\\ & \\
					\AX$\BULTF \BULSFP X \fCenter Y$
					\LL{\fns ss$_L$}
					\doubleLine
					\UI$\BULSTXFP X \fCenter Y$
					\DP
					 &
					\AX$Y \fCenter \BULTF \BULSFP X$
					\RL{\fns ss$_R$}
					\doubleLine
					\UI$Y \fCenter \BULSTXFP X$
					\DP
					\\
					&\\
					\AX$\BULSYTOF \BULY X \fCenter Y$
					\LL{\fns sc$_L$}
					\doubleLine
					\UI$\BULZU \cdots\BULZK \BULTF X \fCenter Y$
					\DP
					&
					\AX$Y \fCenter \BULSYTOF \BULZ X$
					\RL{\fns sc$_R$}
					\doubleLine
					\UI$Y \fCenter \BULZU \cdots\BULZK \BULTF X$
					\DP
						\\
					
				\end{tabular}
			
		\end{center}

	\end{enumerate}

\end{definition}


\section{Properties}\label{sec:folproperties}
In the present section, we outline the proofs of soundness, completeness, conservativity, cut elimination and subformula property of the calculus $\mathrm{D.FO}$.

\subsection{Soundness}\label{ssec:sound}
In the present subsection, we outline the  verification of the soundness of the rules of $\mathrm{D.FO}$ w.r.t.~the semantics of heterogeneous $\mathcal{L}$-algebras  (cf.~Definition \ref{def:heterogeneousalgebras}). As done in analogous situations 
\cite{GP:linear, SDM}, the first step consists in interpreting structural symbols as logical symbols according to their (precedent or succedent) position,\footnote{\label{footnote:def precedent succedent pos}For any  sequent $x\vdash y$, we define the signed generation trees $+x$ and $-y$ by labelling the root of the generation tree of $x$ (resp.\ $y$) with the sign $+$ (resp.\ $-$), and then propagating the sign to all nodes according to the polarity of the coordinate of the connective assigned to each node. Positive (resp.\ negative) coordinates propagate the same (resp.\ opposite) sign to the corresponding child node.  Then, a substructure $z$ in $x\vdash y$ is in {\em precedent} (resp.\ {\em succedent}) {\em position} if the sign of its root node as a subtree of $+x$ or $-y$ is  $+$ (resp.\ $-$).}
as indicated in the synoptic tables of Section \ref{sec:multicalc}. This makes it possible to interpret sequents as inequalities, and rules as quasi-inequalities. For example, the rule on the left-hand side below is interpreted as the bi-implication on the right-hand side:

\begin{center}
	\begin{tabular}{rcl}
		\AX$\BULTFP \DIAY X\fCenter Y$
		\doubleLine
		\UI$\DIAZ \BULZYTFP X \fCenter Y$
		\DisplayProof
		&$\quad\rightsquigarrow\quad$&
		$\forall a\forall b[\bultfp  \diay     a\leq b \Leftrightarrow \diaz \bulzytfp a\leq b]$
	\end{tabular}
\end{center}
Notice that the validity of the bi-implication is equivalent to the validity of the analytic inductive (in fact left-primitive) identity 
$$\bultfp \diay a=\diaz \bulzytfp a$$ 
which holds by definition on every heterogeneous $\mathcal{L}$-algebra. The soundness of the remaining unary rules is proven analogously. As to the rules for substitution (cf.\ Definition \ref{def:FODL}.8), the soundness follows from Lemma \ref{lem:subrules}. 
Finally, the soundness of the rules for equality (cf.\ Definition \ref{def:FODL}.10) follows immediately from Corollary \ref{lem:eqrules}.
\medskip

This completes the proof that $\mathrm{D.FO}$ and $\mathrm{D.FO}^\ast$ are sound with respect to the class of heterogeneous algebraic $\mathcal{L}$-models. Since this class properly includes the structures arising from standard $\mathcal{L}$-models, by Proposition \ref{prop:sematicfaith} $\mathrm{D.FO}$ and $\mathrm{D.FO}^\ast$ are sound with respect to standard $\mathcal{L}$-models.

\subsection{Translations and completeness}

The aim of this subsection is to show the following
\begin{thm}\label{theo:completeness}
	If $A\in\mathcal{L}_{\mathrm{MT}}$ and $\vdash_{\mathrm{FO}}A^\tau$ then $\vdash_{\mathrm{D.FO}}A$.
\end{thm}
We will proceed as follows: For every $A\in\mathcal{L}$ and $F\in\mathcal{P}(\mathsf{Var})$, we will define, in two steps,  a canonical $\mathcal{L}_F$-formula $\kappa(F,A)$, free of explicit substitutions, and show that $A^F\vdash \kappa(F,A)$ and $\kappa(F,A)\vdash A^F$ are derivable sequents in $\mathrm{D.FO}$ for any $A^F\in \mathcal{L}_F$ such that $(A^F)^\tau=A$. Using this, we will show that for any formulas $A,B\in\mathcal{L}_F$, if $A^\tau=B^\tau$ then $A\vdash B$ and $B\vdash A$ are derivable sequents in $\mathrm{D.FO}$. Thanks to this observation, to show Theorem \ref{theo:completeness} it is enough to show that for every $\mathcal{L}$-formula $A$ which is a theorem of first-order logic and every $F\in\mathcal{P}_\omega(\mathsf{Var})$,  some $A'\in\mathcal{L}_F$ exists such that $\vdash_F A' $ is provable in $\mathrm{D.FO}$.

Let us preliminarily show  the following

\begin{lemma}\label{lem:derrules}
	The following rules are derivable:
	
	\begin{center}	
		\begin{tabular}{rl}
			\AX$\BOXSTAR X \ADRARR \DIASTAR Y \fCenter Z$
			\UI$\DIASTAR (X \ADRARR Y) \fCenter Z$
			\DP
			&
			\AX$Z \fCenter \DIASTAR X \ARARR \BOXSTAR Y$
			\UI$Z \fCenter \BOXSTAR (X \ARARR Y)$
			\DP
			\\ & \\
			\AX$\DIASTAR X \ADLARR \BOXSTAR Y \fCenter Z$
			\UI$\DIASTAR (X \ADLARR Y) \fCenter Z$
			\DisplayProof
			& 		
			\AX$Z \fCenter \BOXSTAR X \ALARR \DIASTAR Y$
			\UI$Z \fCenter \BOXSTAR (X \ALARR Y)$
			\DP
			\\
			&  \\
			\AX$\DIASTAR X \AAND \DIASTAR Y \fCenter Z$
			\UI$\DIASTAR (X \AAND Y) \fCenter Z$
			\DP
			&
			\AX$Z \fCenter \BOXSTAR X \AOR \BOXSTAR Y$
			\UI$Z \fCenter \BOXSTAR (X \AOR Y)$
			\DP\\
		\end{tabular}
	\end{center}
	where $\DIASTAR \in \{\DIASF, \DIAX\}$ and $\BOXSTAR \in\{\BOXSF, \BOXX\}$ (see also Proposition \ref{prop:somefolproperties}).
\end{lemma}
\begin{proof}
	The proof uses the rules of Definition \ref{def:FODL}.14. We only prove one, the others being shown similarly:
	\begin{center}
		\AX$X \fCenter \DIAX Y \ARARR \BOXX Z$
		\UI$\DIAX Y \AAND X \fCenter \BOXX Z$
		\UI$\BULX (\DIAX Y \AAND X) \fCenter Z$
		\UI$\BULX \DIAX Y \AAND \BULX X \fCenter Z$
		\UI$\BULX \DIAX Y \fCenter Z \ALARR \BULX X$
		\UI$Y\fCenter Z \ALARR \BULX X$
		\UI$Y \AAND \BULX X \fCenter Z$
		\UI$\BULX X \fCenter Y \ARARR Z$
		\UI$X \fCenter \BOXX (Y \ARARR Z)$
		\DisplayProof
	\end{center}
	
\end{proof}

Let us now define a function $\sigma:\mathcal{L}_{\mathrm{MT}}\to\mathcal{L}_{\mathrm{MT}}$ such that for any formula $A\in\mathcal{L}_{\mathrm{MT}}$ the formula $\sigma(A)$ is free of explicit substitutions. We define $\sigma$ recursively as follows:

\begin{enumerate}
	\item $\sigma(R(t_{\overline{x}}))=R(t_{\overline{x}})$;
	\item $\sigma(B\bullet C)=\sigma(B)\bullet\sigma(C)$ where $\bullet\in\{\aand,\lor,\to\}$;
	\item $\sigma(\forall x B)=\forall x \sigma(B)$ and $\sigma(\exists x B)=\exists x\sigma(B)$;
	\item $\sigma(\bulx   B)=\bulx  \sigma(B)$;
	\item if $A$ is $\bulsf   B$ then $\sigma(A)$ is defined recursively on the complexity of $B$:
	\begin{itemize}
		\item if $B$ is $R(t_{\overline{y}})$ then $\sigma(A)=R(t(s_{\overline{x}}/\overline{x})_{\overline{y}})$;
		\item if $B$ is $C\bullet D$ then $\sigma(A)=\sigma(\bulsf C)\bullet\bulsf D)$ for $\bullet\in\{\aand,\lor,\to\}$;
		\item if $B$ is $\diay C$ then $\sigma(A)=\diaz \sigma(\bulzysf C)$ where $z\notin\mathsf{FV}(B)$;
		\item if $B$ is $\boxy C$ then $\sigma(A)=\boxz \sigma(\bulzysf C)$ where $z\notin\mathsf{FV}(B)$;
		\item if $B$ is $\buly C$ then $\sigma(A)=\bulzu \cdots\bulzk \sigma(\bulsfy  C)$, where $\{z_1,\ldots,z_k\}=\mathsf{FV}(s_y)\setminus\mathsf{FV}(s_{F\setminus\{y\}})$;
		\item if $B$ is $\bulrt C$ then $\sigma(A)=\sigma(\bulrsxt C)$.	
	\end{itemize}
\end{enumerate}

Notice that, while $\tau$ performs the substitutions and removes $\bulx  $-operators, $\sigma$ simply performs the substitutions. Hence applying $\tau$ after $\sigma$ is the same as applying $\tau$. This motivates the following:
\begin{lemma}\label{lem:substipropbasic}
	If $A\in\mathcal{L}_F$ then $\sigma(A)\in\mathcal{L}_F$. Furthermore $(\sigma(A))^\tau=A^\tau$.
\end{lemma}
\begin{proof}Straightforward induction on the complexity of $A$.
\end{proof}

\begin{lemma}\label{lem:subperform}
	For every $A\in\mathcal{L}_{\mathrm{MT}}$, the sequents $A\vdash \sigma(A)$ and $\sigma(A)\vdash A$ are derivable in $\mathrm{D.FO}$.
\end{lemma}
\begin{proof}
	By induction on the complexity of $A$. If $A$ is $R(t_{\overline{x}})$, then $\RTX \vdash R(t_{\overline{x}})$ is an axiom and we have
	\begin{center}
		\AX$\RTX \fCenter R(t_{\overline{x}})$
		\UI$R(t_{\overline{x}})\fCenter R(t_{\overline{x}})$
		\DisplayProof
	\end{center} as required.
	
	\noindent	If $A$ is $B\aand C$ then by induction hypothesis we have:
	\begin{center}
		\AX$B\fCenter\sigma(B)$
		\AX$C\fCenter\sigma(C)$
		\BI$B \AAND C\fCenter\sigma(B)\aand\sigma(C)$
		\UI$B \aand C\fCenter\sigma(B)\aand\sigma(C)$
		\DP
	\end{center}
	and
	\begin{center}
		\AX$\sigma(B) \fCenter B$
		\AX$\sigma(C) \fCenter C$
		\BI$\sigma(B) \AAND \sigma(C) \fCenter B \aand C$
		\UI$\sigma(B) \aand \sigma(C) \fCenter B \aand C$
		\DP
	\end{center}
	which yields the result since $\sigma(B\aand C)=\sigma(B)\aand\sigma(C)$ by definition.

	If $A$ is $B\lor C$ or $B\to C$ we argue similarly.

	If $A$ is $\bulx   B$ then:
	\begin{center}
		\AX$B\fCenter\sigma(B)$
		\UI$\BULX B\fCenter\BULX \sigma(B)$
		\UI$\bulx B\fCenter\BULX \sigma(B)$
		\UI$\bulx B\fCenter\bulx \sigma(B)$
		\DP
	\end{center}
	which yields the result since $\sigma(\bulx  B)=\bulx  \sigma(B)$ (and similarly for the other direction).

	If $A$ is $\boxx B$ then:
	\begin{center}
		\AX$B \fCenter \sigma(B)$
		\UI$\boxx B \fCenter \BOXX \sigma(B)$
		\UI$\boxx B\fCenter\boxx \sigma(B)$
		\DP
	\end{center}
	which yields the result since $\sigma(\boxx B)=\boxx \sigma(B)$,
	and similarly for the other direction and for $\diax B$.

	If $A$ is $\bulsf B$, notice that
	\begin{center}
		\begin{tabular}{lr}
			\AX$\BULSF B \fCenter \sigma(\bulsf B)$
			\UI$\bulsf B \fCenter \sigma(\bulsf B)$
			\DP
			&
			\AX$\sigma(\bulsf B) \fCenter \BULSF B$
			\UI$\sigma(\bulsf B) \fCenter \bulsf B$
			\DP
		\end{tabular}
	\end{center}
	and hence it is enough to show that $\BULSF B\vdash\sigma(\bulsf B)$ and $\sigma(\bulsf B)\vdash\BULSF B$ are derivable. We proceed by induction on the complexity of $B$.
	
	If $B$ is $R(\overline{t}_{\overline{y}})$ then by definition $\sigma(\bulsf R(t_{\overline{y}}))=R(t(s_{\overline{x}}/\overline{x})_{\overline{y}})$ and therefore we have: 
	\begin{center}
		\begin{tabular}{c c c}
			\AX$\hat{R}(t(s_{\overline{x}}/\overline{x})_{\overline{y}}) \fCenter R(t(s_{\overline{x}}/\overline{x})_{\overline{y}})$
			\UI$\BULSF \hat{R}(t_{\overline{y}}) \fCenter R(t(s_{\overline{x}}/\overline{x})_{\overline{y}})$
			\UI$\hat{R}(t_{\overline{y}}) \fCenter \BOXSF R(t(s_{\overline{x}}/\overline{x})_{\overline{y}})$
			\UI$R(t_{\overline{y}}) \fCenter \BOXSF R(t(s_{\overline{x}}/\overline{x})_{\overline{y}})$
			\UI$\BULSF R(\overline{t}_{\overline{x}})\fCenter R(t(s_{\overline{x}}/\overline{x})_{\overline{y}})$
			\DP
			&
			\AX$\hat{R}(t_{\overline{y}}) \fCenter R(t_{\overline{y}})$
			\UI$\BULSF \hat{R}(t_{\overline{y}})\fCenter \BULSF R(t_{\overline{y}})$
			\UI$\hat{R}(t(s_{\overline{x}}/\overline{x})_{\overline{y}}) \fCenter\BULSF R(t_{\overline{y}})$
			\UI$R(t(s_{\overline{x}}/\overline{x})_{\overline{y}})\fCenter\BULSF  R(t_{\overline{y}})$
			\DP
			
		\end{tabular}	
	\end{center}
	which yield the desired result.

	If $B$ is $C\aand D$ then:
	\begin{center}
		\AX$\BULSF C \fCenter \sigma(\bulsf C)$
		\AX$\BULSF D \fCenter \sigma(\bulsf D)$
		\BI$\BULSF C \AAND \BULSF D \fCenter \sigma(\bulsf C)\aand\sigma(\bulsf D)$
		\UI$\BULSF (C \AAND D) \fCenter \sigma(\bulsf C)\aand\bulsf D)$
		\UI$(C \AAND D) \fCenter \BOXSF \sigma(\bulsf C)\aand\sigma(\bulsf D)$
		\UI$(C\aand D) \fCenter \BOXSF \sigma(\{(s_x/x)\}C)\aand\sigma(\bulsf D)$
		\UI$\BULSF (C\aand D) \fCenter \sigma(\bulsf C)\aand\sigma(\bulsf D)$
		\DP
	\end{center}
	and
	\begin{center}
		\AX$\sigma(\bulsf C) \fCenter \BULSF C$
		\UI$\DIASF \sigma(\bulsf C) \fCenter C$
		\AX$\sigma(\bulsf C) \fCenter \BULSF C$
		\UI$\DIASF \sigma(\bulsf D)\fCenter D$
		\BI$\DIASF \sigma(\bulsf C) \AAND \DIASF \sigma(\bulsf D)\fCenter C\aand D$
		\UI$\BULSF (\DIASF \sigma(\bulsf C) \AAND \DIASF \sigma(\bulsf D)) \fCenter \BULSF (C\aand D)$
		\UI$\BULSF \DIASF \sigma(\bulsf C) \AAND \BULSF \DIASF \sigma(\bulsf D) \fCenter \BULSF (C\aand D)$
		\UI$\BULSF \DIASF \sigma(\bulsf C) \fCenter \BULSF \DIASF \sigma(\bulsf D) \ARARR \BULSF (C\aand D)$
		\LL{\fns sadj}
		\UI$\sigma(\bulsf C) \fCenter \BULSF \DIASF \sigma(\bulsf D) \ARARR \BULSF (C\aand D)$
		\UI$\BULSF \DIASF \sigma(\bulsf D)\fCenter \sigma(\bulsf C) \ARARR \BULSF (C\aand D)$
		\LL{\fns sadj}
		\UI$\sigma(\bulsf D)\fCenter \sigma(\bulsf C) \ARARR \BULSF (C\aand D)$
		\UI$\sigma(\bulsf C) \AAND \sigma(\bulsf D)\fCenter \BULSF (C\aand D)$
		\UI$\sigma(\bulsf C) \aand \sigma(\bulsf D)\fCenter \BULSF (C\aand D)$
		\DP
	\end{center}
	and if $B$ is $C\lor D$ then the proof is analogous.
	
	If $B$ is $C\to D$ then:
	\begin{center}
		\AX$\BULSF C\fCenter\sigma(\bulsf C)$
		\AX$\sigma(\bulsf D)\fCenter \BULSF D$
		\BI$\sigma(\bulsf C)\to\sigma(\bulsf D)\fCenter\BULSF C \ARARR \BULSF D$
		\UI$\sigma(\bulsf C)\to\sigma(\bulsf D)\fCenter\BULSF (C \ARARR D)$
		\UI$\DIASF (\sigma(\bulsf C)\to\sigma(\bulsf D))\fCenter C \ARARR D$
		\UI$\DIASF (\sigma(\bulsf C)\to\sigma(\bulsf D))\fCenter C \ararr D$
		\UI$\sigma(\bulsf C)\to\sigma(\bulsf D)\fCenter\BULSF (C \ararr D)$
		\DP
	\end{center}
	and
	\begin{center}
		\AX$\sigma(\bulsf C) \fCenter \BULSF C$
		\UI$\DIASF \sigma(\bulsf C) \fCenter C$
		\AX$\BULSF D \fCenter \sigma(\bulsf D)$
		\UI$D \fCenter \BOXSF \sigma(\bulsf D)$
		\BI$C\to D \fCenter \DIASF \sigma(\bulsf C) \ARARR \BOXSF \sigma(\bulsf D)$
		\UI$\BULSF (C\to D) \fCenter \BULSF (\DIASF \sigma(\bulsf C) \ARARR \BOXSF \sigma(\bulsf D))$
		\UI$\BULSF (C\to D) \fCenter \BULSF \DIASF \sigma(\bulsf C) \ARARR \BULSF \BOXSF \sigma(\bulsf D)$
		\UI$\BULSF (C\to D) \AAND \BULSF \DIASF \sigma(\bulsf C)\fCenter \BULSF \BOXSF \sigma(\bulsf D)$
		\RightLabel{\fns sadj}
		\UI$\BULSF (C\to D) \AAND \BULSF \DIASF\sigma(\bulsf C) \fCenter \sigma(\bulsf D)$
		\UI$\BULSF \DIASF \sigma(\bulsf C)\fCenter \sigma(\bulsf D) \ALARR \BULSF (C\to D)$
		\LL{\fns sadj}
		\UI$\sigma(\bulsf C) \fCenter \sigma(\bulsf D) \ALARR \BULSF (C\to D)$
		\UI$\sigma(\bulsf C) \AAND \BULSF (C\to D) \fCenter \sigma(\bulsf D)$
		\UI$\BULSF (C\to D)\fCenter\sigma(\bulsf C) \ARARR \sigma(\bulsf D)$
		\DP
	\end{center}
	
	If $B$ is $\diay C$:
	\begin{center}
		\AX$\BULZYSF C \fCenter \sigma(\bulzysf C)$
		\UI$\DIAZ \BULZYSF C \fCenter \diaz \sigma(\bulzysf C)$
		\LeftLabel{\fns sq$_L$}
		\UI$\BULSF \DIAY C \fCenter \diaz \sigma(\bulzysf C)$
		\UI$\DIAY C \fCenter \BOXSF \diaz \sigma(\bulzysf C)$
		\UI$\diay C \fCenter \BOXSF \diaz \sigma(\bulzysf C)$
		\UI$\BULSF \diay C \fCenter \diaz \sigma(\bulzysf C)$
		\DP
	\end{center}
	and
	\begin{center}
		\AX$\sigma(\bulzysf C) \fCenter\BULZYSF C$
		\UI$\DIAZYSOF \sigma(\bulzysf C) \fCenter C$
		\UI$\DIAY \DIAZYSOF \sigma(\bulzysf C) \fCenter \diay C$
		\UI$\DIAZYSOF \sigma(\bulzysf C) \fCenter \BULY \diay C$
		\UI$\sigma(\bulzysf C) \fCenter \BULZYSF \BULY \diay C$
		\RL{\fns sc$_R$}
		\UI$\sigma(\bulzysf C) \fCenter\BULZ \BULSF \diay C$
		\UI$\DIAZ \sigma(\bulzysf C) \fCenter\BULSF \diay C$
		\UI$\diaz \sigma(\bulzysf C) \fCenter\BULSF \diay C$
		\DP
	\end{center}
	If $B$ is $\buly C$ then:
	\begin{center}
		\AX$\BULSFY C \fCenter \sigma(\bulsfy C)$
		\UI$\BULZU \cdots \BULZK \BULSFY  C \fCenter \BULZU \cdots \BULZK \sigma(\bulsfy C)$
		\LeftLabel{\fns sc$_L$}
		\UI$\BULSF \BULY C \fCenter \BULZU \cdots \BULZK \sigma(\bulsfy C)$
		\UI$\BULSF \BULY C \fCenter \bulzu \cdots \bulzk \sigma(\bulsfy C)$
		\UI$\BULY C \fCenter \BOXSF \bulzu \cdots \bulzk \sigma(\bulsfy C)$
		\UI$\buly C \fCenter \BOXSF \bulzu \cdots \bulzk \sigma(\bulsfy C)$
		\UI$\BULSF \buly C \fCenter\bulzu \cdots \bulzk \sigma(\bulsfy C)$
		\DP
	\end{center}
	the other direction being symmetrical. Finally if $B$ is $\bulrt C$ then:
	\begin{center}
		\AX$\BULRSXT C \fCenter \sigma(\bulrsxt C)$
		\LeftLabel{\fns ss$_L$}
		\UI$\BULSF \BULRT C \fCenter \sigma(\bulrsxt C)$
		\UI$\BULRT C \fCenter \BOXSF \sigma(\bulrsxt C)$
		\UI$\bulrt C \fCenter \BOXSF \sigma(\bulrsxt C)$
		\UI$\BULSF \bulrt C \fCenter \sigma(\bulrsxt C)$
		\DP
	\end{center}
	
	This concludes the proof.
\end{proof}

Let us define a translation $\kappa:\mathcal{P}_\omega(\mathsf{Var})\times\mathcal{L}\to\mathcal{L}_{\mathrm{MT}}$ such that $\kappa(F,A)\in\mathcal{L}_{F\cup\mathsf{FV}(A)}$, by recursion on $A$ as follows:

\begin{itemize}
	\item $\kappa(F,R(t_{\overline{x}}))=\bulzu \cdots\bulzk  R(t_{\overline{x}})$ where $\{z_1,\ldots,z_k\}=F\setminus\mathsf{FV}(t_{\overline{x}})$;
	\item $\kappa(F,A\,\bullet\, B)=\kappa(F\cup\FV(B),A)\,\bullet\,\kappa(F\cup\FV(A),B))$ where $\bullet\in\{\aand,\lor,\to\}$;
	\item $\kappa(F,\exists xA)=\bulx  \diax \kappa(F,A)$ if $x\in F$ and $\kappa(F,\exists xA)=\diax \kappa(F\cup\{x\},A)$ if $x\notin F$ (and similarly for $\forall x A$);
\end{itemize}


\begin{lemma}\label{lem:canonicalform}
	For every formula $A\in\mathcal{L}_F$ that does not contain explicit substitutions the sequents $\kappa(F,A^\tau)\vdash A$ and $A\vdash\kappa(F,A^\tau)$ are derivable in $\mathrm{D.FO}$.
\end{lemma}
\begin{proof}
	We proceed by induction on the complexity of $A$:
	If $A$ is $R(t_{\overline{x}})$ then $\kappa(F,A^\tau)=A$ and we are done.
	If $A$ is $B\aand C$, notice that $\FV(A)=\mathsf{FV}(B)=\mathsf{FV}(C)=F$. then:
	\begin{center}
		\AX$\kappa(F,B^\tau)\fCenter B$
		\AX$\kappa(F,C^\tau)\fCenter C$
		\BI$\kappa(F,B^\tau)\AAND \kappa(F,C^\tau)\fCenter B\aand C$
		\UI$\kappa(F,B^\tau)\aand\kappa(F,C^\tau)\fCenter B\aand C$
		\DisplayProof
	\end{center}
	We work similarly for the other direction and for $\to$ and $\lor$.
	
	If $A$ is $\diax B$ then $\mathsf{FV}(B)=F\cup\{x\}$:
	
	\begin{center}
		\AX$\kappa(F\cup\{x\},B^\tau)\fCenter B$
		\UI$\DIAX \kappa(F\cup\{x\},B^\tau)\fCenter\diax B$
		\UI$\diax \kappa(F\cup\{x\},B^\tau)\fCenter\diax B$
		\DisplayProof
	\end{center}
	and likewise for the other direction.
	
	Finally assume that $A$ is of the form $\bulxu \cdots \bulxk B$ for some $k\in\omega$, where $B$ is not of the form $\bulz C$. We proceed by induction on the complexity of $B$. We wills show that $\BULXU \cdots\BULXK B\vdash\kappa(F,(\bulxu \cdots\bulxk B)^\tau)$ and $\kappa(F,(\bulxu \cdots\bulxk B)^\tau)\vdash\BULXK \cdots\BULXK B$. Notice  preliminarily that $(\bulxu \cdots\bulxk B)^\tau=B^\tau$ for all $B$.
	
	If $B$ is $R(t_{\overline{x}})$ then $\kappa(F,  B)=\bulxu \cdots\bulxk B$:
	\begin{center}
		\AX$\RTX \fCenter R(t_{\overline{x}})$
		\UI$R(t_{\overline{x}})\fCenter R(t_{\overline{x}})$
		\UI$\BULXU \cdots\BULXK R(t_{\overline{x}})\fCenter\BULXU \cdots\BULXK R(t_{\overline{x}})$
		\UI$\BULXU \cdots\BULXK R(t_{\overline{x}})\fCenter\bulxu \cdots\bulxk R(t_{\overline{x}})$
		\DisplayProof
	\end{center}
	and similarly for the other direction.
	
	If $B$ is $C\aand D$:
	
	\begin{center}
		\AX$\BULXU \cdots \BULXK C \fCenter \kappa(F,C^\tau)$
		\AX$\BULXU \cdots\BULXK D\fCenter\kappa(F,D^\tau)$
		\BI$\BULXU \cdots\BULXK C \AAND \BULXU \cdots\BULXK D \fCenter \kappa(F,C^\tau)\aand\kappa(F,D^\tau)$
		\UI$\BULXU \cdots\BULXK (C \AAND D) \fCenter \kappa(F,C^\tau)\aand\kappa(F,D^\tau)$
		\dashedLine
		\UI$(C \AAND D) \fCenter \BOXXK \cdots \BOXXU (\kappa(F,C^\tau)\aand\kappa(F,D^\tau))$
		\UI$(C\aand D) \fCenter \BOXXK \cdots \BOXXU (\kappa(F,C^\tau)\aand\kappa(F,D^\tau))$
		\UI$\BULXU \cdots\BULXK (C\aand D) \fCenter \BOXXK \cdots \BOXXU (\kappa(F,C^\tau)\aand\kappa(F,D^\tau))$
		\DP
	\end{center}
	and
{\footnotesize 	\begin{center}
		\AXC{$\kappa(F,C^\tau) \fCenter \BULXU \cdots\BULXK C$}
		\UIC{$\DIAXK \cdots \DIAXU \kappa(F,C^\tau)\fCenter C$}
		\AXC{$\kappa(F, D^\tau) \fCenter \BULXU \cdots\BULXK D$}
		\UIC{$\DIAXK \cdots \DIAXU \kappa(F,D^\tau) \fCenter D$}
		\BIC{$\DIAXK \cdots \DIAXU \kappa(F,C^\tau) \AAND \DIAXK \cdots \DIAXU \kappa(F,D^\tau) \fCenter C\aand D$}
		\dashedLine
		\UIC{$\BULXU \cdots\BULXK (\DIAXK \cdots \DIAXU \kappa(F,C^\tau) \AAND \DIAXK \cdots \DIAXU \kappa(F,D^\tau)) \fCenter \BULXU \cdots\BULXK (C\aand D)$}
\UIC{$\BULXU \cdots\BULXK \DIAXK \cdots \DIAXU \kappa(F,C^\tau) \AAND \BULXU \cdots\BULXK \DIAXK \cdots \DIAXU \kappa(F,D^\tau) \fCenter \BULXU \cdots\BULXK (C\aand D)$}
		\dashedLine
		\LL{\fns cadj$\times2k$}
		\UI$\kappa(F,C^\tau)\AAND\kappa(F,D^\tau)\fCenter \BULXU \cdots\BULXK (C\aand D)$
		\UI$\kappa(F,C^\tau)\aand\kappa(F,D^\tau)\fCenter \BULXU \cdots\BULXK (C\aand D)$
		\DP
	\end{center}}
	and if $B$ is $C\lor D$ and $C\to D$ we argue similarly  (see also proof of Lemma \ref{lem:subperform}).

	
	If $B$ is $\diax  C$, let us assume without loss of generality that $x_1=x$\footnote{Notice that if $x\notin\{x_1,\ldots,x_k\}$ the last two steps of the derivations are redundant.}:
	\begin{center}
		\AX$\BULXD \cdots\BULXK C\fCenter\kappa(F, C^\tau)$
		\UI$\DIAX \BULXD \cdots\BULXK C\fCenter\diax \kappa(F, C^\tau)$
		\LeftLabel{\fns cq$_L$}
		\dashedLine
		\UI$\BULXD \cdots\BULXK \DIAX C \fCenter\diax \kappa(F, C^\tau)$
		\dashedLine
		\UI$\DIAX C \fCenter \BOXXK \cdots \BOXXD \diax \kappa(F, C^\tau)$
		\UI$\diax C\fCenter \BOXXK \cdots \BOXXD \diax \kappa(F, C^\tau)$
		\UI$\BULXD \cdots\BULXK \diax C \fCenter \diax \kappa(F, C^\tau)$
		\UI$\BULX \BULXD \cdots\BULXK \diax C\fCenter\BULX \diax \kappa(F, C^\tau)$
		\UI$\BULX \BULXD \cdots\BULXK \diax C\fCenter\bulx  \diax \kappa(F, C^\tau)$
		\DisplayProof
	\end{center}
	and
	\begin{center}
		\AX$\kappa(F, C^\tau)\fCenter \BULXD \cdots\BULXK C$
		\UI$\DIAXK \cdots \DIAXD \kappa(F, C^\tau)\fCenter C$
		\UI$\DIAX \DIAXK \cdots \DIAXD \kappa(F, C^\tau)\fCenter\diax C$
		\dashedLine
		\UI$\kappa(F, C^\tau)\fCenter \BULXD \cdots\BULXK \BULX \diax C$
		\RL{\fns cc$_R$}
		\dashedLine
		\UI$\kappa(F, C^\tau)\fCenter \BULX \BULXD \cdots\BULXK \diax C$
		\UI$\DIAX \kappa(F, C^\tau)\fCenter \BULXD \cdots\BULXK \diax C$
		\UI$\diax \kappa(F, C^\tau)\fCenter \BULXD \cdots\BULXK \diax C$
		\UI$\BULX \diax \kappa(F, C^\tau)\fCenter \BULX \BULXD \cdots\BULXK \diax C$
		\UI$\bulx  \diax \kappa(F, C^\tau)\fCenter \BULX \BULXD \cdots\BULXK \diax C$
		\DP
	\end{center}
	
	For $\forall xC$ the proof is analogous.  This concludes the proof.
\end{proof}

\begin{cor}\label{cor:tocanonical}
	For any formula $A\in\mathcal{L}_F$  the sequents $A\vdash\kappa(F,A^\tau)$ and $\kappa(F,A^\tau)\vdash A$ are derivable in $\mathrm{D.FO}$.
\end{cor}
\begin{proof}
	By Lemmas \ref{lem:subperform} and \ref{lem:canonicalform}, the sequents $A\vdash\sigma(A)$ and  $\sigma(A)\vdash\kappa(F,\sigma(A)^\tau)$ are derivable in $\mathrm{D.FO}$. Hence, by applying Cut,  $A\vdash \kappa(F,\sigma(A)^\tau)$ is derivable in $\mathrm{D.FO}$. Finally, by Lemma \ref{lem:substipropbasic} $\sigma(A)^\tau=A^\tau$. This concludes the proof of the first part of the statement. The second part is argued analogously.
\end{proof}

\begin{cor}\label{cor:keyprovability}
	If $A,B\in\mathcal{L}_F$ and $A^\tau=B^\tau$ then $A\vdash B$ is derivable in $\mathrm{D.FO}$.
\end{cor}
\begin{proof}
	By Corollary \ref{cor:tocanonical} $A\vdash\kappa(F,A^\tau)$ and $\kappa(F,B^\tau)\vdash B$ are derivable. Since $A^\tau=B^\tau$, we derive $A\vdash B$ from these two sequents by applying  Cut.
\end{proof}

The corollary above implies that in order to prove completeness in the sense specified in Theorem \ref{theo:completeness}, it is enough to prove that, for every $\mathcal{L}$-formula $A$ which is a theorem of first-order logic and every $F\in\mathcal{P}_\omega(\mathsf{Var})$, there exists some $A'\in\mathcal{L}_F$ with $(A')^\tau=A$ such that $\vdash_F A' $ is provable in $\mathrm{D.FO}$. In what follows, we will provide the required derivations.

If $A$ is a propositional tautology then $\kappa(F,A)$ is derivable using the propositional fragment of $\mathrm{D.FO}$. The derivations are omitted. If $A$ is of the form $\forall x(B\rightarrow C)\rightarrow(\forall xB\rightarrow \forall xC)$ then let $D:=\kappa(F\cup\{x\},B)$ and $E:=\kappa(F\cup\{x\},C)$:

\begin{center}
	\AX$D \fCenter D$
	\UI$\boxx D \fCenter \BOXX D$
	\UI$\BULX \boxx D \fCenter D$
	
	\AX$E \fCenter E$
	
	\BI$D \ararr E \fCenter \BULX \boxx D \ARARR E$
	\UI$\boxx (D \ararr E) \fCenter \BOXX (\BULX \boxx D \ARARR E)$
	\UI$\BULX \boxx D \ararr E) \fCenter \BULX \boxx D \ARARR E$
	\UI$\BULX \boxx D \AAND \BULX \boxx (D \ararr E) \fCenter E$
	\UI$\BULX (\boxx D \AAND \boxx (D \ararr E)) \fCenter E$
	\UI$\boxx D \AAND \boxx (D \ararr E) \fCenter \BOXX E$
	\UI$\boxx D \AAND \boxx (D \ararr E) \fCenter \boxx E$
	\UI$\boxx (D \ararr E) \fCenter \boxx D \ARARR \boxx E$
	\UI$\boxx (D \ararr E)\fCenter \boxx D \rightarrow \boxx E$
	\UI$\AATOP_{\!F} \fCenter \boxx (D \ararr E) \ARARR (\boxx D \ararr \boxx E)$
	\UI$\AATOP_{\!F} \fCenter \boxx (D \ararr E) \ararr ( \boxx D \ararr \boxx E)$
	\DP
\end{center}
which yields the desired result given that $(\boxx  (D \rightarrow E)\rightarrow(\boxx  D \rightarrow \boxx  E))^\tau=A$. In case $x\in F$ we apply the introduction of $\BULX $ in the penultimate step.

If $A$ is of the form $B\rightarrow\forall xB$ where $x\notin\mathsf{FV}(B)$, then let $\kappa(F\setminus\{x\},B)=C$:

\begin{center}
	\AX$C \fCenter C$
	\UI$\BULX C \fCenter \BULX C$
	\UI$\BULX C \fCenter \bulx C$
	\UI$C \fCenter \BOXX \bulx C$
	\UI$C \fCenter \boxx \bulx C$
	\UI$\AATOP_{\!F} \fCenter C \ARARR \boxx \bulx C$
	\UI$\AATOP_{\!F} \fCenter C \ararr \boxx \bulx C$
	\DP
\end{center}
which yields the desired result given that $(C\rightarrow\boxx \bulx  C)^\tau=A$.
If $A$ is of the form $\forall xB\rightarrow B(t/x)$, where $t$ is free for $x$ in $B$ then let $C=\kappa(\mathsf{FV}(B),B)$ and $(t_x,\overline{y}_{\overline{y}})$ for $\overline{y}=\mathsf{FV}(B)\setminus\{x\}$, and $\sgr{(\overline{y}_{\overline{y}})}$ is the identity substitution. We proceed by cases. If $x\in \mathsf{FV}(B)$:
\begin{center}
\AX$C \fCenter C$
\UI$\BULTXYOY C \fCenter\BULTXYOY C$
\UI$C \fCenter \BOXTXYOY \BULTXYOY C$
\UI$\boxx C \fCenter \BOXX \BOXTXYOY \BULTXYOY C$
\UI$ \BULTXYOY \BULX \boxx C\fCenter\BULTXYOY C$
\LL{\fns sc$_L$}
\UI$\BULZU \cdots\BULZK \BULYOYO \BULYOYO \boxx C \fCenter \BULTXYOY C$
\UI$\bulzu \cdots\bulzk \bulyoyo \boxx C\fCenter\BULTXYOY C$
\UI$\AATOP_{\!F} \fCenter \bulzu \cdots \bulzk \bulyoyo \boxx C \ARARR \BULTXYOY C$
\UI$\AATOP_{\!F} \fCenter \bulzu \cdots \bulzk \bulyoyo \boxx C\ararr \BULTXYOY C$
\DP
\end{center}
which yields the desired result given that $(\bulzu \cdots\bulzk \bulyoyo \boxx C\rightarrow\BULTXYOY C)^\tau=A$. If $x\notin \mathsf{FV}(B)$:
\begin{center}
\AX$C \fCenter C$
\UI$\BULX C\fCenter\BULX C $
\UI$\BULTXYOY \BULX C \fCenter \BULTXYOY \BULX C$
\UI$\BULX C \fCenter \BOXTXYOY \BULTXYOY \BULX C$
\UI$\bulx C \fCenter \BOXTXYOY \BULTXYOY \BULX C$
\UI$\boxx \bulx C \fCenter \BOXX \BOXTXYOY \BULTXYOY \BULX C$
\UI$\BULTXYOY \BULX \boxx \bulx  C\fCenter\BULTXYOY \BULX C$
\LeftLabel{\fns sc$_L$}
\UI$\BULZU \BULZD \cdots\BULZK \BULYOYO \boxx \bulx C\fCenter\BULZU \BULZD \cdots\BULZK \BULYOYO C$
\LL{\fns $\BULZU_M$}
\UI$\BULZD \cdots\BULZK \BULYOYO \boxx \bulx C\fCenter\BULZD \cdots\BULZK \BULYOYO C$
\dashedLine
\UI$\BULYOYO \boxx \bulx C\fCenter\BULYOYO C$
\UI$\bulyoyo \boxx \bulx C\fCenter\bulyoyo C$
\UI$\AATOP_{\!F} \fCenter \bulyoyo \boxx \bulx C \ARARR \bulyoyo C$
\UI$\AATOP_{\!F} \fCenter \bulyoyo \boxx \bulx C\rightarrow\bulyoyo C$
\DP
\end{center}
which yields the desired result given that $(\bulyoyo    \boxx \bulx  C\rightarrow\bulyoyo    C)^\tau=A$.

As for the equality axioms we have
\begin{center}
	\AX$\ETT \fCenter t = t$
	\UI$\AATOP_{\!F}, \ETT \fCenter t = t$
	\UI$\AATOP_{\!F} \fCenter t = t$
	\DP
\end{center}
Let $A$ be a formula with free variable $v$, let $\{x_1,\ldots,x_n\}=\FV(A)\setminus\FV(t,s)$, $\{y_1,\ldots,y_m\}=\FV(s)\setminus\FV(t,A)$ and $\{z_1,\ldots,z_k\}=\FV(t)\setminus\FV(s,A)$, and let's assume without loss of generality that $v\notin\FV(s,t)$. Let $$B=\bulyu \cdots\bulym \bulzu \cdots\bulzk A.$$ Notice that $A=B^\tau$. We have
{\small \begin{center}
	\AX$B \fCenter B$
	\UI$\BULSVZYX B\fCenter \BULSVZYX B$
	\UI$\BULV \BULSVZYX B\fCenter\BULV  \BULSVZYX B$
	\UI$\BULSVZYX \BULV \BULSVZYX B\fCenter\BULSVZYX \BULV  \BULSVZYX B$
	\LeftLabel{\fns sc$_L$}
	\UI$\BULZYX \BULSVZYX B\fCenter\BULSVZYX \BULV  \BULSVZYX B$
	\LeftLabel{\fns ss$_L$}
	\UI$\BULSVZYX B\fCenter\BULSVZYX \BULV  \BULSVZYX B$
	\UI$\BULXU \cdots\BULXN \EST \AAND \BULSVZYX B\fCenter\BULSVZYX \BULV \BULSVZYX B$
	\UI$\BULXU \cdots\BULXN \EST \fCenter \BULSVZYX B \ARARR \BULSVZYX \BULV \BULSVZYX B$
	\UI$\BULXU \cdots\BULXN \EST \fCenter \BULSVZYX (B \ARARR \BULV \BULSVZYX B)$
	\UI$\BULXU \cdots\BULXN \EST \fCenter \BULTVZYX (B \ARARR \BULV \BULSVZYX B)$
	\UI$\BULXU \cdots\BULXN \EST \fCenter\BULTVZYX B \ARARR \BULTVZYX \BULV \BULSVZYX B$
	\UI$\BULXU \cdots\BULXN \EST \AAND \BULTVZYX B\fCenter\BULTVZYX \BULV \BULSVZYX B$
	\RightLabel{sc$_R$}
	\UI$\BULXU \cdots\BULXN \EST \AAND \BULTVZYX B \fCenter \BULZYX  \BULSVZYX B$
	\RightLabel{ss$_R$}
	\UI$\BULXU \cdots\BULXN \EST \AAND \BULTVZYX B \fCenter \BULSVZYX B$
	\dashedLine
	\UI$\bulxu \cdots\bulxn s = t \fCenter \bultvzyx B \ARARR \bulsvzyx B$ 
	\DP
\end{center}}
which yields the desired result given that {\scriptsize$$(\bulxu \cdots\bulxn s=t\,\to\,(\st{s_v,\overline{z}_{\overline{z}},\overline{y}_{\overline{y}},\overline{x}_{\overline{x}}}B\,\to\, \st{s_v,\overline{z}_{\overline{z}},\overline{y}_{\overline{y}},\overline{x}_{\overline{x}}}B))^\tau\,=\,(s=t\to(A(s/y)\to A(t/y))).$$}

Finally using necessitation we obtain the universal closure of tautologies by repeated application of the following pattern:

\begin{center}
	\AX$\AATOP_{\!F\cup\{x\}} \fCenter A$
	\UI$\BULX \AATOP_{\!F} \fCenter A$
	\UI$\AATOP_{\!F} \fCenter \BOXX A$
	\UI$\AATOP_{\!F} \fCenter \boxx A$
	\DP
\end{center}

This concludes the proof of completeness. 

\subsection{Conservativity}
To argue that the calculus $\mathrm{D.FO}$  conservatively extend the axiomatic system $\vdash_{\mathrm{FO}}$, we follow the standard proof strategy discussed in \cite{GMPTZ,linearlogPdisplayed}.  We need to show that, for all formulas $A$ of $\mathcal{L}$, if $\top_F\vdash_F \kappa(F,A)$ is derivable in $\mathrm{D.FO}$,  then  $\vdash_{\mathrm{FO}} A$. This claim can be proved using  the following facts: (a) the rules of $\mathrm{D.FO}$ are sound w.r.t.~~the semantics of heterogeneous $\mathcal{L}$-algebras (cf.~Section \ref{ssec:sound});  (b) The Hilbert-style system $\vdash_{\mathrm{FO}}$ is complete w.r.t.~first-order logic models (cf.\ Section \ref{ssec:folprel}); and (c)  first-order models are equivalently presented as heterogeneous  algebras (cf.~Section \ref{sec:mtfol}), so that the semantic consequence relations arising from each type of structures preserve and reflect the translation (cf.~Proposition \ref{prop:sematicfaith}). Then, let $A$ be an $\mathcal{L}$-formula. If  $\top_F\vdash_F \kappa(F,A)$ is derivable in  $\mathrm{D.FO}$, then, by (a),  $(\mathbb{H},V)\models\top_F\leq_F \kappa(F,A)$, for every heterogeneous algebraic $\mathcal{L}$-model $(\mathbb{H},V)$. By (c), this implies that $M \models A$, where $M$ is a first order model. By (b), this implies that $\vdash_{\mathrm{FO}} A$, as required.

\subsection{Cut elimination and subformula property}

In the present section, we outline the proof of cut elimination and subformula property for the calculus $\mathrm{D.FO}$ of Section \ref{sec:multicalc}. As discussed earlier on, the design of this calculus allows for its cut elimination and subformula property to be inferred from a metatheorem, following the strategy introduced by Belnap for display calculi. The metatheorem to which we will appeal is \cite[Theorem 4.1]{TrendsXIII}  (cf.\ \cite[Section 3]{TrendsXIII}) for the class of \emph{multi-type calculi}, of which  $\mathrm{D.FO}$ is a particularly well-behaved element, since it enjoys the full display property. For this reason, $\mathrm{D.FO}$ satisfies the following more restricted version of condition C$'''_5$:

\noindent \textbf{C$'''_5$: Closure of axioms under cut.} If $x\vdash a$ and  $a\vdash y $  are axioms, then $x\vdash y$ is  again an  axiom.

By this metatheorem, it is enough to verify that $\mathrm{D.FO}$  meets the conditions  listed in \cite[Section 3]{TrendsXIII} with C$'''_5$ modified as indicated above. All conditions except C$_8$ are readily satisfied by inspecting the rules. In what follows we verify C$_8$. We only treat the cases of the atomic formulas and the heterogeneous connectives:
\paragraph*{Atomic formulas:}
We have:
\begin{center}
	\begin{tabular}{ccc}
		\bottomAlignProof
		\AX$\RTX  \fCenter R(\overline{t}_{\overline{x}})$
		\AXC{$\vdots$ \raisebox{1mm}{$\pi_1$}}
		\noLine
		\UI$\RTX  \fCenter Y$
		\UI$ R(\overline{t}_{\overline{x}}) \fCenter Y$
		\BI$\RTX  \fCenter Y$
		\DisplayProof
		
		& $\rightsquigarrow$ &
		
		\bottomAlignProof
		\AXC{ $\vdots$ \raisebox{1mm}{$\pi_1$}}
		\noLine
		\UI$\RTX  \fCenter Y$
		\DisplayProof
		\\
	\end{tabular}
\end{center}

\paragraph*{Quantifiers and their adjoint:}

\begin{center}
	\begin{tabular}{ccc}
		\bottomAlignProof
		\AXC{\ \ \ $\vdots$ \raisebox{1mm}{$\pi_1$}}
		\noLine
		\UI$X \fCenter A$
		\UI$\DIAX X \fCenter \diax A$
		\AXC{\ \ \ $\vdots$ \raisebox{1mm}{$\pi_2$}}
		\noLine
		\UI$\DIAX A \fCenter Y$
		\UI$\diax A \fCenter Y$
		\BI$\DIAX X \fCenter Y$
		\DP
		
		& $\rightsquigarrow$ &
		
		\!\!\!\!\!\!\!
		\bottomAlignProof
		\AXC{\ \ \ $\vdots$ \raisebox{1mm}{$\pi_1$}}
		\noLine
		\UI$X \fCenter A$
		\AXC{\ \ \ $\vdots$ \raisebox{1mm}{$\pi_2$}}
		\noLine
		\UI$\DIAX A \fCenter Y$
		\UI$A \fCenter \BULX Y$
		\BI$X \fCenter\BULX Y$
		\UI$\DIAX X \fCenter Y$
		\DP
		\\
	\end{tabular}
\end{center}

\medskip
\begin{center}
	\begin{tabular}{ccc}
		\bottomAlignProof
		\AXC{\ \ \ $\vdots$ \raisebox{1mm}{$\pi_1$}}
		\noLine
		\UI$X \fCenter \BOXX A$
		\UI$X \fCenter \boxx A$
		\AXC{\ \ \ $\vdots$ \raisebox{1mm}{$\pi_2$}}
		\noLine
		\UI$A \fCenter Y$
		\UI$\boxx A \fCenter \BOXX Y$
		\BI$X \fCenter \BOXX Y$
		\DisplayProof
		
		& $\rightsquigarrow$ &
		
		\!\!\!\!\!\!\!
		\bottomAlignProof
		\AXC{\ \ \ $\vdots$ \raisebox{1mm}{$\pi_1$}}
		\noLine
		\UI$X \fCenter \BOXX A$
		\UI$ \BULX X \fCenter A$
		\AXC{\ \ \ $\vdots$ \raisebox{1mm}{$\pi_2$}}
		\noLine
		\UI$A \fCenter Y$
		\BI$\BULX X \fCenter Y$
		\UI$X\fCenter \BOXX Y$
		\DP
		\\
	\end{tabular}
\end{center}
\medskip
\begin{center}
	\begin{tabular}{ccc}
		\bottomAlignProof
		\AXC{\ \ \ $\vdots$ \raisebox{1mm}{$\pi_1$}}
		\noLine
		\UI$X \fCenter \BULX A$
		\UI$X \fCenter \bulx  A$
		\AXC{\ \ \ $\vdots$ \raisebox{1mm}{$\pi_2$}}
		\noLine
		\UI$\BULX A \fCenter Y$
		\UI$\bulx   A \fCenter Y$
		\BI$X \fCenter Y$
		\DisplayProof
		
		& $\rightsquigarrow$ &
		
		\!\!\!\!\!\!\!
		\bottomAlignProof
		\AXC{\ \ \ $\vdots$ \raisebox{1mm}{$\pi_1$}}
		\noLine
		\UI$X \fCenter \BULX A$
		\UI$\DIAX X \fCenter A$
		\AXC{\ \ \ $\vdots$ \raisebox{1mm}{$\pi_2$}}
		\noLine
		\UI$\BULX A \fCenter Y$
		
		\UI$A \fCenter \BOXX Y$
		\BI$\DIAX X \fCenter \BOXX Y$
		\UI$\BULX \DIAX X \fCenter Y$
		\LeftLabel{\fns cadj}
		\UI$X \fCenter Y $
		\DP
		\\
	\end{tabular}
\end{center}

\noindent The cases for $\bultf  $ is done similarly to the one above.

\section{Conclusions and further directions}\label{sec:condir}

\paragraph{Contributions.}
In this article we have introduced a proper multi-type display calculus for classical first-order logic and shown that it is sound, complete, conservative with respect to the standard classical first-order logic and enjoys cut elimination and the subformula property. We intended  to capture first-order logic with a calculus in which quantifiers are represented also at the structural level and rules are closed under uniform substitution. We achieved this by developing an idea of Wansing's, that the proof-theoretic treatment of quantifiers can emulate that of modal operators, in a multi-type setting in which formulas with different sets of free variables have different types. This multi-type environment is supported semantically by certain classes of heterogeneous algebras in which the maps interpreting the existential and universal quantifiers are the left and the right adjoint respectively of one and the same injective Boolean algebra homomorphism. The same adjunction pattern accounts for the semantics of substitution.

\paragraph{Correspondence theory for first-order logic.}
The semantic analysis in Section \ref{sec:semanalfol} can be regarded as a multi-type and ALBA-powered version of the correspondence results observed in \cite{vbpredicate}. Thanks to the systematic connections established between unified correspondence theory and the theory of analytic display calculi, we are now in a position to apply the results and insights of unified correspondence to the semantic environment of Section \ref{sec:semanalfol} and systematically exploit them for proof-theoretic purposes. This might turn out to be a useful tool in the analysis of generalized quantifiers (see discussion below).

\paragraph{A modular environment.}
Thanks to the fact that, in the environment of this calculus, both substitutions and quantifiers are explicitly represented as logical and structural connectives, we can now explore systematically the space of their properties and their possible interactions. For instance, the rules {\fns sq$_R$} and {\fns sq$_L$} express in a transparent and explicit way the book-keeping concerning variable capturing in first-order logic. More interestingly, this environment allows for a finer-grained analysis of fundamental interactions between quantifiers and intensional connectives. For instance, in the present calculus the rules in Definition \ref{def:FODL}.15 encode the fact that the cylindrification maps are Boolean algebra homomorphisms, which in turn captures the fact that classical propositional connectives are all extensional. If we change the propositional base to e.g.\ intuitionistic or bi-intuitionistic logic, or if we expand classical first-order logic with modal connectives it is desirable to allow for some additional flexibility. For instance the constant domain axiom $\forall x(B(x)\lor A)\to((\forall xB(x))\lor A)$ can be captured by the analytic structural rule below on the left hand side, which is interderivable with the one on the right hand side:
\begin{center}
	\begin{tabular}{lcrr}

		\AX$X \fCenter \BOXX (Y \AOR \BULX Z)$
		\UI$X \fCenter \BOXX Y \AOR Z$
		\DP
		&
		
		&
		\AX$\BULX X \ADRARR \BULX Y \fCenter Z$
		\UI$\BULX (X \ADRARR Y) \fCenter Z$
		\DP
		&($\ast$)	
	\end{tabular}
\end{center}
Indeed:
{\small \begin{center}
	\begin{tabular}{lr}

		\AX$X \fCenter \BOXX (Y \AOR \BULX Z)$
		\UI$\BULX X \fCenter Y \AOR \BULX Z$
		\UI$\BULX Z \ADRARR \BULX X \fCenter Y$
		\UI$\BULX (Z \ADRARR X) \fCenter Y$
		\UI$Z \ADRARR X \fCenter \BOXX Y$
		\UI$X \fCenter \BOXX Y \AOR Z$
		\DP
		&

		\AX$\BULX X \ADRARR \BULX Y \fCenter Z$
		\UI$\BULX Y \fCenter Z \AOR \BULX X$
		\UI$Y \fCenter \BOXX (Z \AOR \BULX X)$
		\UI$Y \fCenter \BOXX Z \AOR X$
		\UI$X \ADRARR Y \fCenter \BOXX Z$
		\UI$\BULX (X \ADRARR Y)\fCenter Z$
		\DP
	\end{tabular}
\end{center}}
Notice that in the hypersequent calculus for bi-intuitionistic first-order logic the constant domain axiom is derivable from the Mix rule, capturing the prelinearity axiom $(A\rightarrow B)\lor(B\to A)$. In the context of this calculus the prelinearity axiom corresponds to the rule
\begin{center}
	\AX$X \fCenter Y$
	\AX$W \fCenter Z$
	\BI$\AATOP_{\!F} \fCenter (X \ARARR Z) \AOR (W \ARARR Y)$
	\DisplayProof
\end{center}
we conjecture that in the present framework the constant domain axiom is not derivable in $\mathrm{D.FO}$ without the two rules in ($\ast$) and with the prelinearity rule above.

\paragraph{A modular environment?}
Notwithstanding the increased modularity that this calculus allows, we feel that something more can be done. For instance, the heterogeneous modal operators use individual variables and terms as parameters, much in the same way in which actions and agents were used as parameters in the first versions of the display calculus for DEL  (cf.\ \cite{GKPLori,FGKPS14a}). While this choice is unproblematic in respect to the cut elimination and actually allows to retrieve it from the metatheorem of \cite{TrendsXIII}, it is also practically cumbersome, since it constrains us to admit as axioms sequents which are not useful when deriving actual formulas (indeed substitution axioms in which structural adjoints $\DIATFP$ and $\BOXTFP$ occur on both sides of the sequent cannot possibly occur in derivations that conclusions of which are pure formulas), and moreover recognizing whether such a sequent is an axiom requires lengthy calculations which should be performed internally to the calculus rather than externally. This is the focus of current investigation. Indeed having axioms and rules such as the following

\begin{center}
	\begin{tabular}{ccccc}

		$x \approx x$
		&&	
		$c\approx c$
		& &
		\AXC{$ t_1\approx t_2$}
		\AXC{$ t_2\approx t_3$}
		\BIC{$ t_1\approx t_3$}
		\DP

	\end{tabular}
\end{center}
\begin{center}
	\begin{tabular}{ccc}
		\AXC{$t_1\approx s_1$}
		\AXC{$\cdots$}
		\AXC{$t_n\approx s_n$}
		\TIC{$f(t_1,\ldots,t_n)\approx f(s_1,\ldots, s_n)$}
		\DP
		&
		\AXC{$t\approx s$}
		\UIC{$t\approx \BULSX x$}
		\DP
		&
		\AXC{$t\approx s$}
		\UIC{$\BULSX x\approx s$}
		\DP
	\end{tabular}
\end{center}
would be much neater. This choice would also be compatible with rules such as
\begin{center}
	\bottomAlignProof
	\AX$\BULX (X \AAND Y) \fCenter Z$
	\UI$\BULX X \AAND \BULX Y \fCenter Z$
	\DP
\end{center}
which are completely unproblematic when $x$ is a parameter as in the present setting but would violate the condition C5 against proliferation when considered a term. However, the metatheorem of \cite{TrendsXIII} allows proliferation limited to ``flat'' types, i.e.\ types the only rules of which are identity and cut. We are working towards a solution which allows to integrate variables and terms as types.

\paragraph{Generalized quantifiers.}
The semantic analysis of Section \ref{sec:semanalfol} shows that quantifiers and substitutions correspond to a very restricted class of maps $f:M^S\to M^T$. Therefore it is natural to ask whether different notions of quantification and substitutions can be investigated in connection with larger or different classes of such functions. This study could be relevant to the analysis of generalized quantifiers in natural language semantics \cite{barwise1981generalized,hintikka1974quantifiers,lewis1975adverbs,lindstrom1966first,westerstaahl1989quantifiers}, to dependence and independence logic \cite{vaananen2007dependence,gradel2013dependence,kontinen2013axiomatizing} and to the analysis of different notions of substitutions \cite{pitts2013nominal,scott2008algebraic,szawiel2014theories}.





\end{document}